\documentclass[10pt,a4paper]{amsart}
\pdfoutput=1
\usepackage[subsec,skipmbb,skipfonts]{rune}

\usepackage[full]{textcomp} 
\usepackage[p,osf]{fbb} 
\usepackage[scaled=.95,type1]{cabin} 
\usepackage[varqu,varl]{zi4} 
\usepackage[T1]{fontenc} 
\usepackage[libertine]{newtxmath}
\usepackage[cal=boondoxo,bb=boondox,frak=boondox]{mathalfa}
\DeclareMathAlphabet{\mathbfsf}{\encodingdefault}{\sfdefault}{bx}{sl}

\renewcommand{\catname}[1]{\textup{\textmd{\textsf{#1}}}}
\renewcommand{\Cat}{\catname{Cat}}
\renewcommand{\Mon}{\catname{Mon}}
\renewcommand{\Seg}{\catname{Seg}}
\renewcommand{\Span}{\catname{Span}}
\newcommand{\CAT}{\catname{CAT}}
\newcommand{\Tw}{\catname{Tw}}

\renewcommand{\Fun}{\catname{Fun}}
\renewcommand{\Map}{\catname{Map}}
\newcommand{\CATI}{\CAT_{\infty}}

\newcommand{\ev}{\txt{ev}}

\newcommand{\xF}{\mathbb{F}}

\newcommand{\Fan}{F_{\txt{an}}}
\newcommand{\Uan}{U_{\txt{an}}}
\newcommand{\Flin}{F_{\txt{lin}}}
\newcommand{\Ulin}{U_{\txt{lin}}}
\theoremstyle{plain}
\newtheorem{thmA}{Theorem}

\title[From analytic monads to $\infty$-operads through Lawvere theories]{From analytic monads to $\infty$-operads\\ through Lawvere theories}

\author{Rune Haugseng}
\address{Norwegian University of Science and Technology (NTNU),
  Trondheim, Norway}
\urladdr{http://folk.ntnu.no/runegha}

\date{\today}

\theoremstyle{definition}

\usepackage{extarrows}

\newcommand{\act}{\txt{act}}

\newcommand{\Yo}{\mathsf{y}}

\newcommand{\PSh}{\mathsf{P}}

\crefformat{equation}{(#2#1#3)}

\let\lim\relax \DeclareMathOperator{\lim}{lim}
\let\colim\relax \DeclareMathOperator{\colim}{colim}

\usepackage[scaled=.98]{BOONDOX-uprscr}

\DeclareSymbolFont{bbold}{U}{bbold}{m}{n}
\DeclareSymbolFontAlphabet{\mathbbold}{bbold}

\renewcommand{\Map}{\textup{\textsf{Map}}}
\renewcommand{\Fun}{\textup{\textsf{Fun}}}

\renewcommand{\Cat}{\textup{\textsf{Cat}}}

\renewcommand{\CatI}{\Cat_{\infty}}

\newcommand{\SpF}{\Span(\xF)}
\newcommand{\SpFL}{\Span(\xF)^{\natural}}
\newcommand{\AnMnd}{\catname{AnMnd}}
\newcommand{\LinMnd}{\catname{LinMnd}}
\newcommand{\AnEnd}{\catname{AnEnd}}
\newcommand{\LinEnd}{\catname{LinEnd}}

\DeclareMathOperator{\Sym}{Sym}

\newcommand{\bfone}{\textbf{\textup{\textlf{1}}}}
\newcommand{\fset}[1]{\textbf{\textup{\textlf{#1}}}}

\newcommand{\Act}{\catname{Act}}

\newcommand{\actto}{\rightsquigarrow}
\newcommand{\intto}{\rightarrowtail}
\newcommand{\xint}{\txt{int}}

\usetikzlibrary{decorations.pathmorphing}  
  \tikzcdset{
    active/.style={rightsquigarrow},
    lactive/.style={leftsquigarrow},    
    inert/.style={tail}
  }

\newcommand{\Lawv}{\catname{LwvTh}}
\newcommand{\LawvSpF}{\Lawv_{/\SpFL}}
\newcommand{\LawvAn}{\LawvSpF^{\txt{an}}}
\newcommand{\LawvAnpre}{\Lawv^{\txt{an}}_{\Th(\Sym)}}
\newcommand{\AlgMnd}{\catname{AlgMnd}}

\newcommand{\Th}{\mathfrak{Th}}
\newcommand{\Md}{\mathfrak{Mod}}

\newcommand{\POpd}{\catname{POpd}}
\newcommand{\PCat}{\catname{PCat}}
\newcommand{\PCatI}{\PCat_{\infty}}

\newcommand{\bbO}{\boldsymbol{\Omega}}
\newcommand{\bbOint}{\boldsymbol{\Omega}^{\txt{int}}}
\newcommand{\bbOel}{\boldsymbol{\Omega}^{\txt{el}}}
\newcommand{\Dint}{\boldsymbol{\Delta}^{\txt{int}}}
\newcommand{\Dintop}{\boldsymbol{\Delta}^{\txt{int},\op}}
\newcommand{\Del}{\boldsymbol{\Delta}^{\txt{el}}}
\newcommand{\Delop}{\boldsymbol{\Delta}^{\txt{el},\op}}
\newcommand{\Oop}{\bbO^{\op}}
\newcommand{\SegD}{\Seg_{\Dop}(\mathcal{S})}
\newcommand{\CSeg}{\catname{CSeg}}
\newcommand{\CSegD}{\CSeg_{\Dop}(\mathcal{S})}
\newcommand{\SegO}{\Seg_{\Oop}(\mathcal{S})}
\newcommand{\CSegO}{\CSeg_{\Oop}(\mathcal{S})}
\begin{document}

\begin{abstract}
  We show that Lurie's model for $\infty$-operads (or more precisely a
  ``flagged'' or ``pinned'' version thereof) is equivalent to the
  analytic monads previously studied by Gepner, Kock, and the author, with an
 $\infty$-operad $\mathcal{O}$ corresponding to the monad for
  $\mathcal{O}$-algebras in spaces. In particular, the $\infty$-operad
  $\mathcal{O}$ is completely determined by this monad. To prove this
  we study the Lawvere
  theories of analytic monads, and show that these are precisely
  pinned \iopds{} in a slight (equivalent) variant of Lurie's
  definition, where finite pointed sets are replaced by spans in
  finite sets.
\end{abstract}

\maketitle

\tableofcontents

\section{Introduction}

The theory of \emph{$\infty$-operads} gives a powerful framework for
working with homotopy-coherent algebraic structures. Our goal in this
paper is to give a direct comparison between Lurie's approach to
\iopds{} \cite{HA} and the \emph{analytic monads} introduced in
\cite{polynomial}. In particular, we will show that an \iopd{}
$\mathcal{O}$ in Lurie's sense is completely determined by the monad
for free $\mathcal{O}$-algebras in spaces.

Before we explain these results in more detail, it is helpful to first
recall the relation between operads in
$\Set$ and monads; 
for simplicity, we focus on the case of one-object
operads\footnote{As opposed to (coloured) operads with more general
  sets of objects/colours.} in $\Set$.
One common description of such
operads is that they are associative
monoids in symmetric sequences with respect to the composition
product. Here a \emph{symmetric sequence} in $\Set$
consists of a sequence of sets
$\mathcal{O}(n)$, $n = 0,1,\ldots$, where $\mathcal{O}(n)$ has an
action of the symmetric group $\Sigma_{n}$; this data can also be
encoded as a functor
$\xF^{\simeq} \to \Set$, where $\xF^{\simeq}$ is the groupoid of
finite sets and bijections. The \emph{composition
  product} of two symmetric sequences $\mathcal{O}$ and $\mathcal{P}$
is then given by the formula
\[ (\mathcal{O} \circ \mathcal{P})(n) \cong \coprod_{k =0}^{\infty}
  \mathcal{O}(k) \times_{\Sigma_{k}} \left(
    \coprod_{i_{1}+\cdots+i_{k} = n} (\mathcal{P}(i_{1}) \times \cdots
    \times \mathcal{P}(i_{k})) \times_{\Sigma_{i_{1}} \times \cdots
      \times \Sigma_{i_{k}}} \Sigma_{n} \right),\] analogous to the
formula for composition of power series.

Joyal~\cite{JoyalAnalytique} extended the analogy with power series by
proving that under left Kan extension along the inclusion
$\xF^{\simeq} \to \Set$, symmetric sequences in $\Set$ (or
\emph{species} in Joyal's terminology) are identified with a certain
class of \emph{analytic} endofunctors of $\Set$. Not all natural
transformations among such functors come from morphisms of symmetric
sequences, but we get an equivalence of categories if we consider a
certain class of ``weakly cartesian'' transformations. Moreover, under
this equivalence the composition product is identified with
composition of endofunctors, which means that one-object operads can
be identified with associative monoids in analytic functors under
composition, \ie{} with the class of \emph{analytic} monads on
$\Set$.\footnote{A similar description applies to operads with an
  arbitrary set of objects, we have only restricted to the one-object
  case for simplicity.}  Moreover, the analytic monad that corresponds
to an operad $\mathcal{O}$ can be identified with the monad for free
$\mathcal{O}$-algebras in $\Set$. 

Analytic monads in the \icatl{} setting were introduced by Gepner, Kock,
and the author in \cite{polynomial}. Here the
appropriate\footnote{For example, appropiate in
  the sense of corresponding to symmetric sequences and their
  generalizations, \cf{} \cite{polynomial}*{\S 3.2}.}
notion of
analytic functor admits a considerably simpler characterization than
in the classical 1-categorical context: if $\mathcal{S}$ is the
\icat{} of spaces or \igpds{}, then a functor
$\mathcal{S}_{/I} \to \mathcal{S}_{/J}$ is \emph{analytic} precisely
when it preserves sifted colimits and weakly contractible
limits. Moreover, the appropriate morphisms between these are
precisely the \emph{cartesian} natural transformations, \ie{} those
whose naturality squares are all cartesian. An analytic monad is then
a monad on a slice $\mathcal{S}_{/I}$ whose underlying endofunctor is
analytic, and whose multiplication and unit transformations are
cartesian. 

In \cite{polynomial}, we defined an \icat{} of analytic monads on
slices $\mathcal{S}_{/I}$ where the base space $I$ can vary, and
related these analytic monads to \iopds{} by proving that this \icat{}
is equivalent to that of \emph{dendroidal Segal spaces} due to
Cisinski and Moerdijk~\cite{CisinskiMoerdijkDendSeg}. These are known
to be equivalent to other models for \iopds{} by the comparison
results of
Cisinski--Moerdijk~\cite{CisinskiMoerdijkDendSeg,CisinskiMoerdijkSimplOpd},
Heuts--Hinich--Moerdijk \cite{HeutsHinichMoerdijkDendrComp},
Barwick~\cite{BarwickOpCat}, Chu, Heuts, and the
author~\cite{iopdcomp}, and most recently Hinich and
Moerdijk~\cite{HinichMoerdijkOpd} (who directly compare Lurie's
\iopds{} to complete dendroidal Segal spaces).

In particular, it follows from these comparisons that analytic monads
are equivalent to Lurie's model for \iopds{}~\cite{HA}, which is by
far the best developed approach. However, the resulting connection between the two 
is rather indirect, since it passes through at least one additional model
for \iopds{}. This is rather unsatisfying, since it seems clear what a
direct equivalence ought to do in one
direction: If $\mathcal{O}$ is an \iopd{}, then we can explicitly
describe the monad for free $\mathcal{O}$-algebras in $\mathcal{S}$,
and this is an analytic monad.

The main goal of the present paper is therefore to prove a direct comparison
between Lurie's \iopds{} and analytic monads (\ie{} without passing
through dendroidal Segal spaces and the other equivalences cited
above), which assigns to an \iopd{} $\mathcal{O}$ the monad for
$\mathcal{O}$-algebras in $\mathcal{S}$. To state this precisely 
we first need to explain a slight wrinkle,
however: we proved in \cite{polynomial} that analytic monads
correspond to dendroidal Segal spaces, but \iopds{} should instead correspond
to the full subcategory of \emph{complete} dendroidal Segal spaces. To
correct for this, we can (following \cite{AyalaFrancisFlagged}) instead
consider what we will called \emph{pinned \iopds{}}: pairs
$(\mathcal{O}, p)$ where $\mathcal{O}$ is an \iopd{} and
$p \colon X \twoheadrightarrow \mathcal{O}_{\angled{1}}^{\simeq}$ is
an essentially surjective morphism from an \igpd{} $X$ to the \igpd{} of objects
of the \iopd{} $\mathcal{O}$. The composite
\begin{equation}
  \label{eq:pinnedmonad}
  \Alg_{\mathcal{O}}(\mathcal{S}) \to
  \Fun(\mathcal{O}_{\angled{1}}^{\simeq}, \mathcal{S}) \xto{p^{*}}
  \Fun(X, \mathcal{S}) 
\end{equation}
is then also a monadic right adjoint corresponding to an analytic
monad. We can now state our main result as follows:
\begin{thmA}\label{thm:main}
  Let $\POpd(\xF_{*})$ denote the \icat{} of pinned \iopds{} and
  $\AnMnd$ that of analytic monads. Then the functor
  \[ \POpd(\xF_{*}) \to \AnMnd, \]
  taking a pinned \iopd{} $(\mathcal{O}, p)$ to the monad
  \cref{eq:pinnedmonad}, is an equivalence.
\end{thmA}

To prove this, we also want an explicit functor in the other
direction, extracting a pinned
\iopd{} from an analytic monad. To this end, we will study \emph{Lawvere
  theories} of analytic monads. Let us say that a monad $T$ on
$\Fun(X, \mathcal{S})$ for some \igpd{} $X$ is \emph{algebraic} if the
functor $T$ preserves sifted colimits; then the \emph{Lawvere theory}
$(\mathcal{L}(T),p)$ of $T$
is given by taking $\mathcal{L}(T)^{\op}$ to be the full subcategory
of $\Alg(T)$ spanned by the free algebras on finite coproducts of
representables in $\Fun(X, \mathcal{S})$, together with the functor $p
\colon X \to \mathcal{L}(T)$ taking a point of $X$ to the free algebra
on the presheaf represented by $x$. More generally, a Lawvere theory
can be defined as an \icat{} $\mathcal{L}$ with finite products,
together with a functor $p \colon X \to \mathcal{L}$ such that every
object of $\mathcal{L}$ is a finite product of objects in the image of
$p$. Generalizing a result of
Gepner--Groth--Nikolaus~\cite{GepnerGrothNikolaus} (for the case $X
\simeq *$), we prove an \icatl{} version of the monad--theory
correspondence for multi-sorted Lawvere theories:
\begin{thmA}\label{thm:lawvere}
  Let $\AlgMnd$ be the \icat{} of algebraic monads and $\Lawv$ that of
  Lawvere theories. Taking Lawvere theories of algebraic monads gives
  a functor \[\Th \colon \AlgMnd \to \Lawv.\] This functor is an
  equivalence, with inverse the functor \[\Md \colon \Lawv \to \AlgMnd\]
  that assigns to a Lawvere theory $(\mathcal{L},p)$ its \icat{}
  $\Mod_{\mathcal{L}}(\mathcal{S})$ of \emph{models} in $\mathcal{S}$,
  meaning functors $\mathcal{L} \to \mathcal{S}$ that preserve finite
  products, together with the functor to $\Fun(X, \mathcal{S})$ given
  by restriction along $p$ (which is a monadic right adjoint
  corresponding to an algebraic monad).
\end{thmA}

We can regard analytic monads as forming a (non-full) subcategory of
algebraic monads, and therefore \cref{thm:lawvere} identifies $\AnMnd$
with a (non-full) subcategory of $\Lawv$. For ordinary categories (and
one-sorted Lawvere theories) this subcategory was identified by
Szawiel--Zawadowski~\cite{SzZ}, but this is not quite the relevant
identification for us: The \icat{} $\AnMnd$ has a terminal object
$\Sym$, which is the monad for commutative monoids, and it turns out
that $\AnMnd$ is a \emph{full} subcategory of algebraic monads over
$\Sym$. The Lawvere theory of $\Sym$ can be identified with the pair
$\SpFL := (\SpF, \{\fset{1}\})$ consisting of the $(2,1)$-category
$\SpF$ of spans of finite sets, together with the inclusion of the
one-element set $\fset{1}$. From \cref{thm:lawvere} we then get an
identification between $\AnMnd$ and a full subcategory $\LawvAn$ of
the \icat{} $\LawvSpF$ of Lawvere theories over $\SpFL$. This \icat{}
turns out to have a very nice alternative description:

The (2,1)-category $\SpF$ has many of the same features as $\xF_{*}$
(which can be regarded as the subcategory containing only those spans
where the backwards map is injective), so that we can modify Lurie's
definition of \iopds{} over $\xF_{*}$ to get a notion of
\emph{$\SpF$-\iopds{}} as certain \icats{} over $\SpF$. A key
difference is that $\SpF$-\iopds{} turn out to have finite products,
so that we can regard pinned $\SpF$-\iopds{} as Lawvere theories. This
gives in fact a fully faithful functor from the \icat{} $\POpd(\SpF)$
of pinned $\SpF$-\iopds{} to $\LawvSpF$,
and just as for pinned \iopds{} over $\xF_{*}$, the corresponding
monads are analytic. 
We thus have an inclusion
\begin{equation}
  \label{eq:popdlawvincl}
 \POpd(\SpF) \hookrightarrow \LawvAn.
\end{equation}
Conversely, we prove the following:
\begin{thmA}\label{thm:SpopdLawv}
  The Lawvere theory of any analytic monad is a pinned
  $\SpF$-\iopd{}.
\end{thmA}
The functor \cref{eq:popdlawvincl} is thus essentially surjective, and
combining this with \cref{thm:lawvere} we have equivalences
\[ \POpd(\SpF) \simeq \LawvAn \simeq \AnMnd.\]
To deduce \cref{thm:main} from this we only need to know that the two
version of \iopds{} are equivalent. This has already been shown as
part of joint work with
Barkan and Steinebrunner~\cite{BHS}, using the symmetric monoidal
envelopes over $\xF_{*}$ and $\SpF$. Specifically,
\cite{BHS}*{Corollary 5.1.13 and Corollary 5.3.17} imply that pulling
back along the inclusion $\xF_{*} \hookrightarrow \SpF$
gives an equivalence between (pinned) \iopds{} over $\SpF$ and
$\xF_{*}$, and this is compatible with \icats{} of algebras.
  
Our last main result is that when we combine \cref{thm:main} with the equivalence
between analytic monads and dendroidal Segal spaces from
\cite{polynomial}, we get the expected relation between complete
dendroidal Segal spaces and \iopds{}:
\begin{thmA}\label{thm:csegO}
  Under the composite equivalence
  \[ \POpd(\xF_{*}) \simeq \AnMnd \simeq
    \Seg_{\bbO^{\op}}(\mathcal{S}) \]
  the full subcategory $\CSeg_{\bbO^{\op}}(\mathcal{S})$ of complete
  dendroidal Segal spaces corresponds to the \icat{} $\Opd(\xF_{*})$
  of \iopds{} (identified as those pinned \iopds{} where the morphism
  of spaces is an equivalence).
\end{thmA}

\subsubsection*{Overview}
\label{sec:overview}

In \S\ref{sec:iopds} we first introduce (pinned) \iopds{} over $\xF_{*}$ and
$\SpF$, and recall the comparison between them. We also introduce
Lawvere theories and their models, and show that pinned
$\SpF$-\iopds{} are examples of Lawvere theories, and their models are
precisely algebras (or monoids) in the \icat{} of spaces.

In \S\ref{sec:Lawveremonads}, we introduce algebraic monads and their
associated Lawvere theories, before proving \cref{thm:lawvere}.

We then study analytic monads and their Lawvere theories in
\S\ref{sec:anmnds}. Here we also show that pinned $\SpF$-\iopds{}
give analytic monads, and prove \cref{thm:SpopdLawv}.

Finally, in \S\ref{sec:completedend} our goal is to prove
\cref{thm:csegO}; this turns out to require an understanding of how
the equivalences between pinned \iopds{}, analytic monads, and
dendroidal Segal spaces restrict to equivalences between full
subcategories of pinned \icats{}, linear monads, and Segal spaces on
$\simp$.

\subsubsection*{Notation}
Much of the notation used is hopefully fairly standard, but we
mention a couple of points here for the reader's convenience:
  \begin{itemize}
\item If $T$ is a monad on an \icat{} $\mathcal{C}$, we will generally denote
  the corresponding monadic adjunction by
  \[ F_{T} : \mathcal{C} \rightleftarrows \Alg(T) : U_{T}.\]
\item If $\mathcal{C}$ is a locally small \icat{}, we denote the
  Yoneda embedding for $\mathcal{C}$ by $\Yo_{\mathcal{C}} \colon
  \mathcal{C} \to \Fun(\mathcal{C}^{\op}, \mathcal{S})$. If the
  \icat{}
  $\mathcal{C}$ is clear from context, we will also just write $\Yo$
  for $\Yo_{\mathcal{C}}$.
\item We write $\PSh(\mathcal{C})$ for the \icat{}
  $\Fun(\mathcal{C}^{\op}, \mathcal{S})$ of presheaves on
  $\mathcal{C}$. We write $\PSh^{*}$ for the contravariant presheaves
  functor and $\PSh_{!}$ for the covariant version (obtained from
  $\PSh^{*}$ by taking left adjoints). We will use without comment
  that the Yoneda embedding gives a natural transformation $\Yo \colon
  \id \to \PSh_{!}$, \cf{} \cite{HHLN2}*{\S 8}.
\end{itemize}

\subsubsection*{Acknowledgments}
\label{sec:acknowledgments}

I thank Hongyi Chu, David Gepner, and Joachim Kock for helpful
discussions related to this paper.

\section{$\infty$-operads and Lawvere theories}\label{sec:iopds}

Our main concern in this section is \iopds{} over $\SpF$ and their
relation to Lurie's \iopds{} and to Lawvere theories. In
\S\ref{sec:iopds-over-xf} we introduce (pinned) \iopds{} over both $\SpF$ and
$\xF_{*}$ and recall the equivalence between them, while in
\S\ref{sec:spf-iopds-as} we define Lawvere theories and show that
pinned $\SpF$-\iopds{} are examples of these. Finally, in the brief
\S\ref{sec:algmodel} we show that if $\mathcal{O}$ is an
$\SpF$-\iopd{}, then $\mathcal{O}$-algebras or $\mathcal{O}$-monoids
in an \icat{} with finite products are the same thing as \emph{models}
of $\mathcal{O}$ when it is viewed as a Lawvere theory.

\subsection{\iopds{} over $\xF_*$ and $\SpF$}
\label{sec:iopds-over-xf}

In this subsection we review the notions of \iopds{} over $\xF_{*}$ and
$\SpF$, and the comparison between them proved in \cite{BHS}. We start
by recalling Lurie's definition of \iopds{} over $\xF_{*}$ from \cite{HA}:

\begin{notation}
  Let $\xF$ denote the ordinary category of finite sets and $\xF_{*}$
  that of finite pointed sets.  We write
  $S_{+} := (S \amalg \{*\}, *)$ for the pointed set obtained by
  adding a disjoint base point $*$ to a set $S$; all objects of $\xF_{*}$ are
  of this form. A morphism $\phi \colon S_{+} \to T_{+}$ is called
  \emph{inert} if $|\phi^{-1}(t)|=1$ for $t \neq *$, or in other words if $\phi$ restricts to an isomorphism $S\setminus \phi^{-1}(*) \to T$,
  and \emph{active}
  if $\phi^{-1}(*) := \{*\}$, \ie{} if $\phi$ restricts to a morphism of sets $S \to T$; we write $\xF_{*}^{\xint}$ and
  $\xF_{*}^{\act}$ for the wide subcategories of $\xF_{*}$ containing
  the inert and active morphisms, respectively. Then
  $(\xF_{*}, \xF_{*}^{\xint}, \xF_{*}^{\act})$ is a factorization
  system: for every morphism $\phi$ of $\xF_{*}$, the groupoid of
  factorizations of $\phi$ as an inert morphism followed by an active
  morphism is contractible. For $s \in S$ we write
  $\rho_{s} \colon S_{+} \to \bfone_{+}$ for the inert map that sends
  all elements of $S$ except $s$ to the base point.
\end{notation}

\begin{defn}\label{defn:F*opd}
  An \emph{\iopd{}} is a functor of \icats{} $p \colon \mathcal{O} \to
  \xF_{*}$ such that
  \begin{enumerate}[(1)]
  \item $\mathcal{O}$ has all $p$-cocartesian lifts of inert morphisms
    in $\xF_{*}$.
  \item Given $X, Y \in \mathcal{O}$ over $S_{+},T_{+} \in \xF_{*}$
    and cocartesian lifts $Y \to Y_{t}$ over $\rho_{t}$, the
    commutative square
    \[
      \begin{tikzcd}
        \Map_{\mathcal{O}}(X,Y) \arrow{r} \arrow{d} & \prod_{t \in
          T}\Map_{\mathcal{O}}(X,Y_{t})\arrow{d} \\
        \Map_{\xF_{*}}(S_{+}, T_{+}) \arrow{r} & \prod_{t \in T} \Map_{\xF_{*}}(S_{+},\bfone_{+})
      \end{tikzcd}
    \]
    is cartesian.
  \item For every $S_{+}\in \xF_{*}$, the functor
    \[ \mathcal{O}_{S_{+}} \to \prod_{s \in S}
      \mathcal{O}_{\bfone_{+}} \]
    given by cocartesian transport along the maps $\rho_{s}$ ($s \in
    S$) is an equivalence.
  \end{enumerate}
\end{defn}

The definition of \iopd{} does not use anything special about
$\xF_{*}$ except for the inert--active factorization system and the
object $\bfone_{+}$ --- as spelled out in \cite{patterns1}*{\S
  9} it can be straightforwardly extended to any \icat{} equipped with
a factorization system and a collection of special ``elementary''
objects. Here we are interested in the variant where $\xF_{*}$ is
replaced by the (2,1)-category of \emph{spans} of finite sets:

\begin{defn}
  We write $\SpF$ for the \icat{} (in fact a (2,1)-category) of spans of
  finite sets, defined as in \cite{spans} or \cite{BarwickMackey}.
  Informally, this has finite sets as objects with the morphisms given
  by \emph{spans}, that is diagrams of the form
  \[
    \begin{tikzcd}
      {} & X \arrow{dl} \arrow{dr} \\
      S & & T,
    \end{tikzcd}
  \]
  with composition given by taking pullbacks. More formally, $\SpF$
  can be defined as the complete Segal space given by
  \[ \Map([n], \SpF) \simeq \Map_{\txt{cart}}(\Tw([n]), \xF);\]
  here $\Tw([n])$ is the \emph{twisted arrow category} of $[n]$,
  which can be succinctly described as the partially ordered set of pairs
  $(i,j)$, $0 \leq i \leq j \leq n$, with
  \[ (i,j) \leq (i',j') \quad \iff \quad i \leq i' \leq j' \leq j, \]
  and $\Map_{\txt{cart}}(\Tw([n]), \xF)$ denotes the subspace of
  $\Map(\Tw([n]), \xF)$ spanned by functors $F \colon \Tw([n]) \to
  \xF$ such that $F$ takes all squares of the form
  \[
    \begin{tikzcd}
      (i,j) \arrow{r} \arrow{d} & (i',j) \arrow{d} \\
      (i,j') \arrow{r} & (i',j') 
    \end{tikzcd}
  \]
  to cartesian squares in $\xF$. 
\end{defn}  

\begin{defn}\label{defn:SpFactint}
  We say a morphism $S \xfrom{f} X \xto{g} T$ in $\SpF$ is \emph{active} if $f$
  is an isomorphism, and \emph{inert} if $g$ is an isomorphism. The
  inert and active morphisms are closed under composition, and if we
  write $\SpF^{\xint}$ and $\SpF^{\act}$ for the wide subcategories of
  $\SpF$ containing only these morphisms, then
  it is easy to see (\cf{} \cite{HHLN2}*{Proposition 2.15}) that we have equivalences
  \[ \SpF^{\xint} \simeq \xF^{\op}, \qquad \SpF^{\act} \simeq \xF.\]  
  For $s \in S$, let $\rho'_{s}$ denote the (inert) span $S \hookleftarrow
  \{s\} \isoto \bfone$.
\end{defn}

\begin{observation}
  The inert and active morphisms form a factorization system on
  $\SpF$; indeed, the analogous claim holds for any \icat{} of spans,
  \cf{} \cite{HHLN2}*{Proposition 4.9}.
\end{observation}

We can then consider the following variant of \cref{defn:F*opd}:
\begin{defn}\label{defn:SpFiopd}
  A \emph{$\SpF$-\iopd{}} is a functor of \icats{} $p \colon \mathcal{O} \to
  \SpF$ such that
  \begin{enumerate}[(1)]
  \item\label{wsf1} $\mathcal{O}$ has all $p$-cocartesian lifts of inert morphisms
    in $\SpF$.
  \item\label{wsf2} Given $X, Y \in \mathcal{O}$ over $S,T \in \xF_{*}$
    and cocartesian lifts $Y \to Y_{t}$ over $\rho'_{t}$, the
    commutative square
    \[
      \begin{tikzcd}
        \Map_{\mathcal{O}}(X,Y) \arrow{r} \arrow{d} & \prod_{t \in
          T}\Map_{\mathcal{O}}(X,Y_{t})\arrow{d} \\
        \Map_{\SpF}(S, T) \arrow{r} & \prod_{t \in T} \Map_{\SpF}(S,\bfone)
      \end{tikzcd}
    \]
    is cartesian.
  \item\label{wsf3} For every $S \in \SpF$, the functor
    \[ \mathcal{O}_{S} \to \prod_{s \in S}
      \mathcal{O}_{\bfone} \]
    given by cocartesian transport along the maps $\rho'_{s}$ ($s \in
    S$) is an equivalence.
  \end{enumerate}
\end{defn}

\begin{observation}\label{rmk:wsfcondred}
  Condition \ref{wsf3} is slightly redundant: Given \ref{wsf1} and
  \ref{wsf2}, we have for all $X,X' \in \mathcal{O}_{\fset{n}}$ equivalences
  \[ \Map_{\mathcal{O}_{\fset{n}}}(X',X) \isoto
    \prod_{i=1}^{n} 
    \Map_{\mathcal{O}}(X', X_{i})_{\rho'_{i}} \isofrom
    \prod_{i=1}^{n}
    \Map_{\mathcal{O}_{\bfone}}(X'_{i}, X_{i}),\]
  where the first map is an equivalence of fibres from one of the
  cartesian squares in \ref{wsf2} and the second is an equivalence by
  the universal property of the cocartesian maps $X' \to
  X'_{i}$. This shows that the functor $\mathcal{O}_{\fset{n}} \to
  \prod_{i=1}^{n} \mathcal{O}_{\bfone}$ in \ref{wsf3} is fully faithful. To conclude that it is an
  equivalence it therefore suffices to additionally assume that these
  functors are essentially surjective.
\end{observation}

\begin{defn}
  For $\mathfrak{F} = \xF_{*},\SpF$ and
  $p \colon \mathcal{O} \to \mathfrak{F}$ an $\mathfrak{F}$-\iopd{},
  we say that a morphism $\phi \colon X \to Y$ in $\mathcal{O}$ is
  \emph{inert} if it is cocartesian and its image in $\mathfrak{F}$ is
  inert, and \emph{active} if its image in $\mathfrak{F}$ is
  active. The inert and active morphisms then form a factorization
  system on $\mathcal{O}$ by \cite{HA}*{Proposition 2.1.2.5}.
\end{defn}
  
\begin{defn}
  For $\mathfrak{F} = \xF_{*},\SpF$, a \emph{morphism of
    $\mathfrak{F}$-\iopds{}} is a commutative triangle
  \[
    \begin{tikzcd}
      \mathcal{O} \arrow{rr} \arrow{dr} & & \mathcal{P} \arrow{dl} \\
       & \mathfrak{F},
    \end{tikzcd}
  \]
  where the two diagonal maps are $\mathfrak{F}$-\iopds{} and the
  horizontal map preserves inert morphisms. We
  define $\Opd(\mathfrak{F})$ to be the subcategory of
  $\Cat_{\infty/\mathfrak{F}}$ spanned by the $\mathfrak{F}$-\iopds{}
  and the morphisms thereof. For $\mathfrak{F}$-\iopds{} $\mathcal{O}$
  and $\mathcal{P}$, we also write $\Alg_{\mathcal{O}}(\mathcal{P})$
  for the full subcategory of
  $\Fun_{/\mathfrak{F}}(\mathcal{O},\mathcal{P})$ spanned by the
  morphisms of $\mathfrak{F}$-\iopds{}.
\end{defn}

\begin{observation}\label{rmk:F*asspan}
  We can think of $\xF_{*}$ as the wide subcategory of $\SpF$
 that contains only those morphisms
    \[
    \begin{tikzcd}
      {} & X \arrow[hookrightarrow]{dl} \arrow{dr}{f} \\
      S & & T
    \end{tikzcd}
  \]
  where the backwards map is injective: we identify this with the map
  of pointed sets $S_{+} \to T_{+}$ given by
  \[ s \mapsto
    \begin{cases}
      *, & s \notin X,\\
      f(s), & s \in X.
    \end{cases}
  \]
  We thus have an inclusion
  $\mathfrak{i} \colon \xF_{*} \to \SpF$, and the factorization system
  on $\xF_{*}$ is obtained by restricting that on $\SpF$ along $\mathfrak{i}$.
\end{observation}

\begin{thm}[\cite{BHS}*{Corollaries 5.1.13 and 5.3.17}]
  Pullback along $\mathfrak{i}$ induces an equivalence
  \[ \mathfrak{i}^{*} \colon \Opd(\SpF) \isoto \Opd(\xF_{*}).\] Moreover,
  for any $\mathcal{O},\mathcal{P} \in \Opd(\SpF)$ we also get a natural
  equivalence
  \[ \Alg_{\mathcal{O}}(\mathcal{P}) \isoto
    \Alg_{\mathfrak{i}^{*}\mathcal{O}}(\mathfrak{i}^{*}\mathcal{P})\]
  between \icats{} of algebras.\qed
\end{thm}

\begin{defn}
  For $\mathfrak{F}$ either $\SpF$ or $\xF_{*}$ (which we think of as
  a subcategory of $\SpF$), we say that a \emph{pinned}
  $\mathfrak{F}$-\iopd{} is an $\mathfrak{F}$-\iopd{} $\mathcal{O}$
  together with a morphism $p \colon X \to \mathcal{O}^{\simeq}_{\bfone}$
  that is essentially surjective (\ie{} surjective on $\pi_{0}$); we
  refer to the morphism $p$ as a \emph{pinning} (of $\mathcal{O}$). We then
  define the \icat{} $\POpd(\mathfrak{F})$ as the pullback
  \[
    \begin{tikzcd}
      \POpd(\mathfrak{F}) \arrow{r} \arrow{d} & \Opd(\mathfrak{F})
      \arrow{d}{(\blank)^{\simeq}_{\bfone}} \\
      \Fun([1], \mathcal{S})_{\txt{es}} \arrow{r}{\ev_{1}} & \mathcal{S},
    \end{tikzcd}
  \]
  where $\Fun([1], \mathcal{S})_{\txt{es}}$ is the full subcategory of
  $\Fun([1], \mathcal{S})$ spanned by the essentially surjective morphisms.
\end{defn}

\begin{notation}
  If $\mathcal{O}$ is an $\mathfrak{F}$-\iopd{}, we write
  $\mathcal{O}^{\natural}$ for the pinned $\mathfrak{F}$-\iopd{}
  $(\mathcal{O}, \mathcal{O}_{\fset{1}}^{\simeq} \xto{=}
  \mathcal{O}_{\fset{1}}^{\simeq})$. (It is easy to see that this
  gives a functor $(\blank)^{\natural} \colon \Opd(\mathfrak{F}) \to
  \POpd(\mathfrak{F})$, which is right adjoint to the functor that
  forgets the pinning.)
\end{notation}

\begin{remark}
  The term \emph{pinned} is taken from \cite{BKW}. The analogous
  notion for $(\infty,n)$-categories has previously been studied by
  Ayala and Francis~\cite{AyalaFrancisFlagged} under the name
  \emph{flagged $(\infty,n)$-categories}. Their terminology is
  inspired by that of ``flags'' in vector spaces, which makes sense
  for general $n$ where one considers a sequence of
  $(\infty,k)$-categories for $k = 0,1,\ldots,n$. Since we only
  consider the case $n = 1$, however, there is not much of a flag to
  speak of, which is why we have chosen to use different terminology.
\end{remark}

Base change along $\mathfrak{i}$ is compatible with the pullback
squares definining pinned \iopds{}, giving:
\begin{cor}
  Pullback along $\mathfrak{i}$ induces an equivalence
  \[ \mathfrak{i}^{*} \colon \POpd(\SpF) \isoto \POpd(\xF_{*})\]
  between \icats{} of pinned \iopds{}. \qed
\end{cor}

\subsection{$\SpF$-\iopds{} and Lawvere theories}
\label{sec:spf-iopds-as}

Our goal in this subsection is to show that we can think of
$\SpF$-\iopds{} as \emph{Lawvere theories} in the following sense:
\begin{defn}
  A \emph{Lawvere theory} $(\mathcal{L},p)$ is an \icat{}
  $\mathcal{L}$ that has finite products, together with a functor
  $p \colon X \to \mathcal{L}$ from an \igpd{} $X$, such that every
  object in $\mathcal{L}$ can be written as a product of objects in
  the image of $p$.
\end{defn}

\begin{observation}
  If we freely adjoin finite coproducts to an \igpd{} $X$, we get the
  \icat{} $\xF_{/X} := \xF \times_{\mathcal{S}} \mathcal{S}_{/X}$ of
  finite sets with a map to $X$.
  If $\mathcal{L}$ is an \icat{} with finite products, a morphism $p
  \colon X \to \mathcal{L}$ therefore uniquely extends to a
  product-preserving functor $\overline{p} \colon \xF_{/X}^{\op} \to
  \mathcal{L}$, and the requirement for $(\mathcal{L},p)$ to be a
  Lawvere theory can be formulated as: $\overline{p}$ is essentially surjective.
\end{observation}

\begin{remark}
  Our definition of Lawvere theory is an \icatl{} version of the
  classical notion of a (finitary) multi-sorted Lawvere theory. The
  one-sorted version was first introduced in Lawvere's thesis~\cite{LawvereThesis}, and has
  been studied for \icats{} by Cranch~\cite{CranchThesis},
  Gepner--Groth-- Nikolaus~\cite{GepnerGrothNikolaus} and
  Berman~\cite{BermanLawv}. There is a very substantial literature on
  1-categorical Lawvere theories and generalizations thereof, which we
  will not attempt to survey; in the $\infty$-categorical setting,
  general notions of algebraic theories have recently been developed
  by Henry--Meadows~\cite{HenryMeadows} and
  Kositsyn~\cite{KositsynTheories}.
\end{remark}

\begin{defn}
  A \emph{morphism of Lawvere theories} from $(\mathcal{L}, p \colon X
  \to \mathcal{L})$ to $(\mathcal{L}',p' \colon X' \to \mathcal{L}')$
  is a product-preserving functor $F \colon \mathcal{L} \to
  \mathcal{L}'$ together with a commutative square
  \[
    \begin{tikzcd}
      X \arrow{r}{f} \arrow{d}{p} & X' \arrow{d}{p'} \\
      \mathcal{L} \arrow{r}{F} & \mathcal{L}'.
    \end{tikzcd}
  \]
  We write $\Lawv$ for the \icat{} of Lawvere theories,
  defined as a subcategory of $\Fun([1], \CatI)$.
\end{defn}

\begin{lemma}\label{lem:SpFprod}
  The disjoint union of finite sets gives cartesian products in
  $\SpF$. In particular, $\SpFL := (\SpF, \{\bfone\} \hookrightarrow
  \SpF)$ is a Lawvere theory.
\end{lemma}
\begin{proof}
  For $S,T \in \xF$, pullback along the inclusions $S,T \hookrightarrow S \amalg T$ gives an equivalence
  \[ \xF_{/S \amalg T} \isoto \xF_{/S} \times \xF_{/T}.\]
  Since we have
  $\Map_{\SpF}(S',S) \simeq
  \xF_{/S' \times S}^{\simeq}$, we get an equivalence
  \[ \Map_{\SpF}(S',S \amalg T) \isoto \Map_{\SpF}(S',S) \times
    \Map_{\SpF}(S',T)\] given by composition with the maps
  $S \xfrom{=} S \hookrightarrow S \amalg T$ and
  $T \xfrom{=} T \hookrightarrow S \amalg T$. These maps therefore
  exhibit $S \amalg T$ as the product of $S$ and $T$ in $\SpF$.
\end{proof}

\begin{propn}\label{SpFopdcond}
  A functor $\pi \colon \mathcal{O} \to \SpF$ is a $\SpF$-\iopd{}
  \IFF{} the following conditions hold:
  \begin{enumerate}[(1$'$)]
  \item\label{spfwsf1} $\mathcal{O}$ has $\pi$-cocartesian morphisms
    over inert morphisms in $\SpF$.
  \item\label{spfwsf2} For $X \in \mathcal{O}$ over $\fset{n}$ in $\SpF$, the
    $\pi$-cocartesian morphisms $X \to X_{i}$ over $\rho'_{i} \colon
    \fset{n} \intto \fset{1}$ exhibit $X$ as the product
    $\prod_{i=1}^{n} X_{i}$ in $\mathcal{O}$.
  \item\label{spfwsf3} The functor $\mathcal{O}_{\fset{n}} \to
    \prod_{i=1}^{n} \mathcal{O}_{\fset{1}}$, given by cocartesian
    transport over the maps $\rho'_{i}$, is essentially
    surjective.
  \end{enumerate}
\end{propn}
\begin{proof}
  We check that conditions \ref{wsf1}--\ref{wsf3} in \cref{defn:SpFiopd}
  correspond to the given conditions
  \ref{spfwsf1}--\ref{spfwsf3}. Here \ref{spfwsf1} is identical to
  \ref{wsf1}. For \ref{wsf2}, consider the commutative
  square
  \[
    \begin{tikzcd}
      \Map_{\mathcal{O}}(X',X) \arrow{r} \arrow{d} & \prod_{i=1}^{n}
      \Map_{\mathcal{O}}(X', X_{i}) \arrow{d} \\
      \Map_{\SpF}(\fset{n}', \fset{n}) \arrow{r}{\sim} &
      \prod_{i=1}^{n} \Map_{\SpF}(\fset{n}', \fset{1}).
    \end{tikzcd}
  \]
  where $X$ and $X'$ are objects of $\mathcal{O}$ over $\fset{n}$ and
  $\fset{n}'$, respectively, and $X \to X_{i}$ is cocartesian over
  $\rho'_{i} \colon \fset{n} \intto \fset{1}$. Here the bottom
  horizontal map is an equivalence, since the maps $\rho'_{i}$ exhibit
  $\fset{n}$ as the product of $n$ copies of $\fset{1}$ in $\SpF$ by
  \cref{lem:SpFprod}. This means the square is cartesian \IFF{} the
  top horizontal morphism is an equivalence, which is true for all
  $X'$ \IFF{} the cocartesian morphisms $X \to X_{i}$ exhibit $X$ as
  the product $\prod_{i=1}^{n} X_{i}$ in $\mathcal{O}$. Thus
  \ref{wsf2} is equivalent to \ref{spfwsf2}. By \cref{rmk:wsfcondred}
  we know that we can replace \ref{wsf3} by the additional
  assumption that the map
  \[ \mathcal{O}_{\fset{n}} \to \prod_{i=1}^{n}
    \mathcal{O}_{\fset{1}} \]
  is essentially surjective, which is precisely \ref{spfwsf3}.
\end{proof}

\begin{lemma}\label{lem:SpFopdprod}
  Suppose $\pi \colon \mathcal{O} \to \SpF$ is a $\SpF$-\iopd{}.
  Given objects $X_{i} \in \mathcal{O}_{S_{i}}$ for $i = 1,\ldots,n$
  then a collection of morphisms $p_{i} \colon X \to X_{i}$ in
  $\mathcal{O}$ exhibit $X$ as the product $\prod_{i} X_{i}$ \IFF{}
  \begin{enumerate}[(i)]
  \item $\pi(p_{i})$ exhibit $\pi(X)$ as the product of $S_{i}$
    in $\SpF$ (so that they are the inert morphisms
    corresponding to the
    inclusions of $S_{i}$ in $\coprod_{i} S_{i}$),
  \item $p_{i}$ is cocartesian for each $i$.
  \end{enumerate}
  In particular, $\mathcal{O}$ has finite products, $\pi$ preserves
  these, and $(\mathcal{O}, \mathcal{O}_{\bfone}^{\simeq})$ is a
  Lawvere theory.
\end{lemma}
\begin{proof}
  Since products are unique up to equivalence if they exist, it
  suffices to show there exists a unique object satisfying the
  conditions, and this is a product.

  Let $\eta_{j}$ denote the inert morphism $S := \coprod_{i} S_{i} \intto
  S_{j}$ in $\SpF$ corresponding to the inclusion of $S_{j}$ in the
  coproduct. Then cocartesian transport along the maps $\eta_{j}$ fits
  in a commutative square
  \[
    \begin{tikzcd}
      \mathcal{O}_{S} \arrow{r}{(\eta_{i,!})_{i}}
      \arrow{d}{\sim} &  \prod_{i} \mathcal{O}_{S_{i}} 
      \arrow{d}{\sim} \\
       \prod_{S} \mathcal{O}_{\bfone} \arrow{r}{\sim} &
       \prod_{i} \left(\prod_{S_{i}} \mathcal{O}_{\bfone}\right),
    \end{tikzcd}
  \]
  where the vertical morphisms are equivalences by condition
  \ref{wsf3} in \cref{defn:SpFiopd}. Then the top horizontal functor
  is also an equivalence, which implies there exists a unique object
  $X \in \mathcal{O}_{S}$ with cocartesian morphisms $X \intto X_{i}$
  over $\eta_{i}$. It remains to show that these exhibit $X$ as a
  product.

  For this, observe that if $X_{i} \intto X_{i,s}$ denotes the
  cocartesian morphisms over $\rho'_{s} \colon S_{i} \intto \bfone$,
  then for any $Y \in \mathcal{O}$ we have a
  commutative square
  \[
    \begin{tikzcd}
      \Map_{\mathcal{O}}(Y, X) \arrow{r} \arrow{d}{\sim} & \prod_{i}
      \Map_{\mathcal{O}}(Y, X_{i}) \arrow{d}{\sim} \\
      \prod_{(i,s) \in S} \Map_{\mathcal{O}}(Y, X_{i,s})
      \arrow{r}{\sim} & \prod_{i} \left(\prod_{s \in S_{i}} \Map_{\mathcal{O}}(Y,X_{i,s})\right),
    \end{tikzcd}
  \]
  where the vertical maps are equivalences by \cref{SpFopdcond}.
  Hence the top horizontal map is also an equivalence, which shows
  that $X$ is a product, as required. It only remains to observe that
  $(\mathcal{O}, \mathcal{O}_{\bfone}^{\simeq})$ is a Lawvere theory
  since every object is a product of objects in $\mathcal{O}_{\bfone}$
  by \cref{SpFopdcond}.
\end{proof}

\begin{observation}
  More generally, any pinned $\SpF$-\iopd{} $(\mathcal{O}, q \colon X
  \twoheadrightarrow \mathcal{O}_{\bfone}^{\simeq})$ is a Lawvere theory.
\end{observation}

\begin{lemma}\label{lem:inertprod}
  Let $\pi \colon \mathcal{O} \to \SpF$ be a $\SpF$-\iopd{}. A
  morphism $\phi \colon X \to X'$ in $\mathcal{O}$
  is inert \IFF{} the composite $X \to X' \intto X'_{i}$ is inert for
  all $i \in \pi(X')$,
  where $X' \to X'_{i}$  is cocartesian over $\rho'_{i}$.
\end{lemma}
\begin{proof}
  Since inert morphisms are closed under composition, the ``only if''
  direction is immediate. In the ``if'' direction, we first check that
  a morphism
  \[ S \xfrom{f} T \xto{g} S'\]
  in $\SpF$ is inert \IFF{} its composite with $\rho'_{i}$ is inert for each $i
  \in S'$. This composite is the span
  \[ S \from T_{i} \to \bfone,\]
  where $T_{i}$ is the fibre of $g$ at $i$. This is inert for all $i$
  \IFF{} each of the fibres $T_{i}$ has a single element, which is
  equivalent to $g$ being an isomorphism, \ie{} to the original span
  being inert, as required.

  Thus if $\phi$ satisfies the condition then $\pi(\phi)$ is inert in
  $\SpF$. If $X \intto X'' \actto X'$ is the inert--active
  factorization of $\phi$, then it follows that $X'' \actto X'$ lies
  over $\id_{\pi(X')}$, and since we have a factorization system the
  map $\phi$ is inert \IFF{} this active map is an equivalence. Now
  observe that the inert--active factorization of the composite
  $X \to X'_{i}$ yields a commutative diagram
  \[
    \begin{tikzcd}
      X \arrow[inert]{r}  \arrow[inert]{dr} & X'' \arrow[active]{r}
      \arrow[inert]{d} & X' \arrow[inert]{d} \\
       & X''_{i} \arrow[active]{r} & X'_{i},
    \end{tikzcd}
  \]
  where by assumption the map $X \to X'_{i}$ is inert, and so $X''_{i}
  \actto X'_{i}$ is an equivalence. But the equivalence of \icats{} $\mathcal{O}_{\pi(X')} \simeq
  (\mathcal{O}_{\bfone})^{\times |\pi(X')|}$ shows that $X'' \actto X'$ is an
  equivalence \IFF{} $X''_{i} \to X'_{i}$ is one for each $i$.
\end{proof}

\begin{propn}\label{propn:SpFopdmor}
  Given a commutative triangle
  \[
    \begin{tikzcd}
      \mathcal{O} \arrow{dr}[swap]{\pi} \arrow{rr}{f} & & \mathcal{O}'
      \arrow{dl}{\pi'} \\
       & \SpF
    \end{tikzcd}
  \]
  where $\mathcal{O}$ and $\mathcal{O}'$ are $\SpF$-\iopds{}, the
  functor $f$ preserves inert morphisms \IFF{} it preserves finite
  products.
\end{propn}
\begin{proof}
  Let us first suppose that $f$ preserves finite products.  To show
  that $f$ preserves inert morphisms, we start by observing that $f$
  preserves inert morphisms over $\rho'_{i}$: Consider
  $X \in \mathcal{O}_{\fset{n}}$, with inert morphisms
  $\xi_{i} \colon X \intto X_{i}$ over $\rho'_{i}$. Then by assumption
  the morphisms $f(\xi_{i}) \colon f(X) \to f(X_{i})$ exhibit $f(X)$
  as the product of the $f(X_{i})$'s, and so are indeed inert by
  \cref{lem:SpFopdprod}. 
  Now consider an arbitrary inert morphism $X \intto X'$ in
  $\mathcal{O}$. To show that $f(X) \to f(X')$ is inert, it suffices
  by \cref{lem:inertprod}
  to show that $f(X) \to f(X') \intto f(X')_{i}$ is inert for each
  $i$, where $f(X') \to f(X')_{i}$ is cocartesian over $\rho'_{i}$. But we know $f(X') \intto f(X')_{i}$ is the image of the inert
  morphism $X' \intto X'_{i}$, and hence  $f(X) \to f(X')_{i}$ is
  indeed inert, since it is the image of the inert composite $X \intto
  X' \intto X'_{i}$. 

  Conversely, suppose we know that $f$ preserves inert morphisms.
  Then \cref{lem:SpFopdprod} implies that $f$ preserves products,
  since these are characterized in terms of inert morphisms.
\end{proof}

Putting the preceding results together, we have shown:
\begin{cor}\label{cor:POpdSpFLT}
  There is a fully faithful functor $\POpd(\SpF) \hookrightarrow
  \Lawv_{/\SpFL}$. \qed
\end{cor}

\subsection{Algebras, monoids and models}\label{sec:algmodel}
If $\mathcal{O}$ is a $\SpF$-\iopd{} and $\mathcal{C}$ is an \icat{}
with finite products, our goal in this subsection is to identify three
structures in $\mathcal{C}$ parametrized by $\mathcal{O}$:
$\mathcal{O}$-algebras in the symmetric monoidal \icat{} given by the
cartesian product on $\mathcal{C}$, $\mathcal{O}$-monoids in
$\mathcal{C}$, and models for $\mathcal{O}$ as a Lawvere theory in
$\mathcal{C}$. 

We begin by introducing the definitions of the last two structures:

\begin{defn}
  If $\mathcal{L}$ and $\mathcal{C}$ are \icats{} with finite products
  (such as the underlying \icat{} of a Lawvere theory), then a
  \emph{model} of $\mathcal{L}$ in $\mathcal{C}$ is a functor
  $\mathcal{L} \to \mathcal{C}$ that preserves finite products. We
  write $\Mod_{\mathcal{L}}(\mathcal{C})$ for the full subcategory of
  $\Fun(\mathcal{L}, \mathcal{C})$ spanned by the models.
\end{defn}

\begin{defn}
  If $\mathcal{O}$ is an $\SpF$-\iopd{} and $\mathcal{C}$ is an
  \icat{} with finite products, then an \emph{$\mathcal{O}$-monoid} in
  $\mathcal{C}$ is a functor $M \colon \mathcal{O} \to \mathcal{C}$
  such that for every $X \in \mathcal{O}$ over $\fset{n} \in \SpF$,
  the morphism
  \[ M(X) \to \prod_{i = 1}^{n} M(X_{i}),\]
  induced by the inert morphisms $X \intto X_{i}$ over $\rho'_{i}$, is
  an equivalence. We write $\Mon_{\mathcal{O}}(\mathcal{C})$ for the
  full subcategory of $\Fun(\mathcal{O}, \mathcal{C})$ spanned by the $\mathcal{O}$-monoids.
\end{defn}

\begin{propn}\label{propn:SpFopdprod}
  If $\mathcal{O}$ is an $\SpF$-\iopd{} and $\mathcal{C}$ is an
  \icat{}, then a functor $M \colon \mathcal{O} \to \mathcal{C}$ is an
  $\mathcal{O}$-monoid \IFF{} $M$ preserves finite products. In
  particular, $\Mon_{\mathcal{O}}(\mathcal{C}) \simeq
  \Mod_{\mathcal{O}}(\mathcal{C})$ as full subcategories of
  $\Fun(\mathcal{O}, \mathcal{C})$.
\end{propn}
\begin{proof}
  It follows from \ref{spfwsf2} in \cref{SpFopdcond} that if $M$ is an
  $\mathcal{O}$-model, then $M$ is an $\mathcal{O}$-monoid. To prove
  the converse, we again use the description of products in
  $\mathcal{O}$ from \cref{lem:SpFopdprod}: Suppose the (necessarily inert) morphisms
  $X \intto X_{i}$ for $i = 1,\ldots,n$ exhibit $X$ in $\mathcal{O}$
  over $S$ as a product of
  the objects $X_{i}$, which lie over $S_{i}$. If
  $X_{i} \intto X_{i,s}$ denotes the
  cocartesian morphisms over $\rho'_{s} \colon S_{i} \intto \bfone$,
  then we have a
  commutative square
  \[
    \begin{tikzcd}
      M(X) \arrow{r} \arrow{d}{\sim} & \prod_{i} M(X_{i})
      \arrow{d}{\sim} \\
      \prod_{(i,s) \in S} M(X_{i,s})
      \arrow{r}{\sim} & \prod_{i} \left(\prod_{s \in S_{i}} M(X_{i,s})\right).
    \end{tikzcd}
  \]
  It follows that the top horizontal morphism is also an equivalence,
  so that $M$ preserves this product.
\end{proof}

\begin{propn}
  Let $\mathcal{C}$ be an \icat{} with finite products. Then there
  exists a $\SpF$-\iopd{} $\mathcal{C}^{\times} \to \SpF$ (in fact a
  cocartesian fibration) with a (product-preserving) functor $\mathcal{C}^{\times} \to
  \mathcal{C}$ such composition with this functor induces an
  equivalence
  \[ \Alg_{\mathcal{O}}(\mathcal{C}^{\times}) \isoto
    \Mon_{\mathcal{O}}(\mathcal{C})\]
  for every $\SpF$-\iopd{} $\mathcal{O}$.
\end{propn}
\begin{proof}
  This follows by the same proof as for the version with $\xF_{*}$ in
  place of $\SpF$, proved in \cite{HA}*{\S 2.4.1}, and we do not
  repeat it here.
\end{proof}

\section{Lawvere theories and algebraic
  monads}\label{sec:Lawveremonads}

It was first shown by Linton~\cite{Linton} that one-sorted Lawvere
theories are equivalent to finitary monads on the category of sets,
and the \icatl{} analogue of this correspondence has been proved by
Gepner--Groth--Nikolaus~\cite{GepnerGrothNikolaus}. In this section we
will extend this result to our multi-sorted notion of Lawvere
theories, proving \cref{thm:lawvere}. We start by introducing the
relevant notion of \emph{algebraic} monad and showing that any Lawvere
theory gives rise to one of these in \S\ref{sec:lawvere-theories}. In
\S\ref{sec:theoryfrommonad} we then define the Lawvere theory
associated to an algebraic monad and prove that its models recover the
monad we started with.

\subsection{Algebraic monads from Lawvere theories}
\label{sec:lawvere-theories}

Our goal in this subsection is to show that models for a Lawvere
theory in the \icat{} of spaces always gives an algebraic monad, in
the following sense:
\begin{defn}
  An \emph{algebraic monad} is a monad $T$ on
  $\Fun(X, \mathcal{S}) \simeq \mathcal{S}_{/X}$ for some
  $X \in \mathcal{S}$ such that the underlying endofunctor of $T$
  preserves sifted colimits.
\end{defn}

\begin{remark}
  In the 1-categorical context, Lawvere theories are related to
  \emph{finitary} monads, \ie{} monads that preserve \emph{filtered}
  colimits. This is not the case in our \icatl{} setting, where the
  monads are required to preserve \emph{sifted} colimits (which in
  $\mathcal{S}_{/X}$ boils down to filtered colimits and simplicial
  colimits). The fundamental reason for this difference is that $\Set$
  is obtained from the category $\xF$ of finite sets by freely adding
  filtered colimits, but $\mathcal{S}$ is obtained from $\xF$ by
  freely adding sifted colimits. Here the relevance of $\xF$ is that
  in both cases we want Lawvere theories to be defined in terms of
  finite products rather than more general limits.
\end{remark}

We note the following general result about monads and colimits,
taken from Henry and Meadows~\cite{HenryMeadows}:
\begin{propn}\label{propn:monadcolim}
  Let $T$ be a monad on an \icat{} $\mathcal{C}$, which has all
  $\mathcal{I}$-shaped colimits for some \icat{} $\mathcal{I}$. Then
  the following are equivalent:
  \begin{enumerate}[(1)]
  \item\label{it:TIcolim} The endofunctor $T \colon \mathcal{C} \to \mathcal{C}$
    preserves $\mathcal{I}$-shaped colimits.
  \item\label{it:Algcolim} The \icat{} $\Alg(T)$ has $\mathcal{I}$-shaped colimits and
    the functor $U_{T}\colon \Alg(T) \to \mathcal{C}$ preserves these.
  \end{enumerate}
\end{propn}
\begin{proof}
  Since $T \simeq U_{T}F_{T}$ where $F_{T}$ is a left adjoint, the
  implication from \ref{it:Algcolim} to \ref{it:TIcolim} is
  immediate. The non-trivial direction is \cite{HenryMeadows}*{Lemma
    6.6}, which in turn is an application of \cite{HA}*{Corollary 4.2.3.5}.
\end{proof}

\begin{cor}\label{lem:algmndUsifted}
  A monad $T$ on $\mathcal{S}_{/X}$ is algebraic \IFF{} for the corresponding monadic
  adjunction
  \[ F_{T} \colon \mathcal{S}_{/X} \rightleftarrows \Alg(T) :\!
    U_{T},\]
  the \icat{} $\Alg(T)$ has sifted colimits and 
  the right adjoint $U_{T}$ preserves these. \qed
\end{cor}

\begin{observation}\label{rmk:Alglocsmall}
  If $T$ is an algebraic monad on $\Fun(X, \mathcal{S})$, then
  $\Alg(T)$ is locally small. Indeed, if $T$ is any monad on a locally
  small \icat{} $\mathcal{C}$ then $\Alg(T)$ is locally small: any
  object $A \in \Alg(T)$ can be written as the simplicial colimit of a
  free resolution, which means that $\Map_{\Alg(T)}(A, B)$ is a
  cosimplicial limit of mapping spaces in $\mathcal{C}$ and so is small.
\end{observation}

\begin{propn}\label{propn:ModisAlgMnd}
  Let $(\mathcal{L}, p \colon X \twoheadrightarrow \mathcal{L})$ be a Lawvere theory. Then:
  \begin{enumerate}[(i)]
  \item\label{it:Modpres} The \icat{} $\Mod_{\mathcal{L}}(\mathcal{S})$ is presentable,
    and the inclusion $\Mod_{\mathcal{L}}(\mathcal{S}) \hookrightarrow
    \Fun(\mathcal{L}, \mathcal{S})$ has a left adjoint.
  \item\label{it:p*MRA} The functor $p^{*}\colon
    \Mod_{\mathcal{L}}(\mathcal{S}) \to \Fun(X, \mathcal{S})$ is a
    monadic right adjoint.
  \item\label{it:TLalgc}  The corresponding monad $T_{\mathcal{L}}$
    is algebraic.
  \item\label{it:LYox} The left adjoint $L$ of $p^{*}$ satisfies
    \[ L \Yo_{X} x \simeq \Yo_{\mathcal{L}^{\op}}(px).\]
  \end{enumerate}
\end{propn}
\begin{proof}
  The \icat{} $\Mod_{\mathcal{L}}(\mathcal{S})$ can be defined as the
  full subcategory of $\Fun(\mathcal{L}, \mathcal{S})$ spanned by the
  objects that are local with respect to the set of maps of the form
  \[ \coprod_{i=1}^{n} \Yo_{\mathcal{L}^{\op}}(px_{i}) \to
    \Yo_{\mathcal{L}^{\op}}\left(\prod_{i=1}^{n} px_{i}\right) \]
  for $x_{1},\ldots,x_{n}$ in $X$. Thus $\Mod_{\mathcal{L}}(\mathcal{S})$ is an
  accessible localization of $\Fun(\mathcal{L}, \mathcal{S})$, which
  means the
  inclusion has a left adjoint $L_{\txt{prod}}$ and
  $\Mod_{\mathcal{L}}(\mathcal{S})$ is presentable. This proves
  \ref{it:Modpres}.

  If $p_{!} \colon
  \Fun(X, \mathcal{S}) \to \Fun(\mathcal{L}, \mathcal{S})$ denotes
  left Kan extension along $p$, then the composite
  $L_{\txt{prod}}p_{!}$ is a left adjoint to $p^{*}$. To prove that
  this is the monadic adjunction for an algebraic monad, 
  it suffices by the
  \icatl{} monadicity theorem, \cite[Theorem 4.7.3.5]{HA}, to prove
  that $p^{*}$ is conservative and preserves sifted colimits.

  To see that $p^{*}$ is conservative, observe that if $\eta \colon F
  \to G$ is a natural transformation between product-preserving
  functors $\mathcal{L} \to \mathcal{S}$, then since every object of
  $\mathcal{L}$ is a product of objects in the image of $p$, the
  component $\eta_{A}$ must be an equivalence for every $A \in
  \mathcal{L}$ if it is an equivalence for objects in the image of
  $p$, \ie{} if $p^{*}\eta$ is an equivalence.

  If we view $p^{*}$ as a functor $\Fun(\mathcal{L}, \mathcal{S}) \to \Fun(X,
  \mathcal{S})$, it preserves all colimits, so to show $p^{*}$
  preserves sifted colimits it suffices to check that
  $\Mod_{\mathcal{L}}(\mathcal{S})$ is closed under sifted colimits in
  $\Fun(\mathcal{L}, \mathcal{S})$. In other words, given a sifted
  diagram $\phi \colon \mathcal{I} \to \Fun(\mathcal{L},
  \mathcal{S})$ where $\phi(i)$ lies in
  $\Mod_{\mathcal{L}}(\mathcal{S})$, so does the colimit of
  $\phi$. This holds because colimits in $\Fun(\mathcal{L},
  \mathcal{S})$ are computed pointwise, and for a product $A \times B$
  in $\mathcal{L}$ we have
  \[ \colim_{i \in \mathcal{I}} \phi_{i}(A \times B) \isoto \colim_{i
      \in \mathcal{I}} \phi_{i}(A) \times \phi_{i}(B) \isoto (\colim_{i
      \in \mathcal{I}} \phi_{i}(A)) \times  (\colim_{i
      \in \mathcal{I}} \phi_{i}(B)),\]
  where the second equivalence holds since sifted colimits by
  definition commute with finite products in $\mathcal{S}$. This
  proves \ref{it:p*MRA} and \ref{it:TLalgc}

  It remains to prove \ref{it:LYox}. We saw that the left adjoint of
  $p^{*}$ is
  $L_{\txt{prod}}p_{!}$, and we know $p_{!}\Yo_{X}x \simeq
  \Yo_{\mathcal{L}^{\op}}p(x)$, so this amounts to checking that
  $\Yo_{\mathcal{L}^{\op}}p(x)$ already lies in the full subcategory
  $\Mod_{\mathcal{L}}(\mathcal{S})$. In other words, we must show that
  the functor
  \[ \Map_{\mathcal{L}}(p(x), \blank) \colon \mathcal{L} \to
    \mathcal{S}\] preserves finite products, which is obvious.
\end{proof}

\begin{defn}
  The appropriate notion of \emph{morphism of algebraic monads} for us
  is a lax morphism of monads, \ie{} a lax natural transformation $S
  \to T$ between algebraic monads, viewed as functors of
  $(\infty,2)$-categories from the universal monad 2-category to
  $\CATI$, the $(\infty,2)$-category of \icats{}. As shown in
  \cite{adjmnd}, such a lax morphism is equivalent to a commutative
  square
  \begin{equation}
    \label{eq:algmndmor}
    \begin{tikzcd}
      \Alg(T) \arrow{r}{F^{*}} \arrow{d}{U_{T}} & \Alg(T) \arrow{d}{U_{S}} \\
      \Fun(X, \mathcal{S}) \arrow{r}{f^{*}} & \Fun(Y,\mathcal{S}),
    \end{tikzcd}
  \end{equation}
  where in our case we want $f^{*}$ to be given by composition with a
  morphism $f \colon Y \to X$ of \igpds{}. We can thus define an
  \icat{} $\AlgMnd^{\op}$ of algebraic monads as the full subcategory
  of the pullback of $\ev_{1}\colon \Fun([1], \LCatI) \to \LCatI$
  along
  $\Fun(\blank, \mathcal{S}) \colon \mathcal{S}^{\op} \to \LCatI$
  whose objects are the monadic right adjoints corresponding to
  algebraic monads. Equivalently, by \cref{lem:algmndUsifted} these
  are the right adjoints that are conservative and preserve sifted
  colimits.
\end{defn}

\begin{observation}\label{rmk:AlgMorpressifted}
  Given a morphism of algebraic monads as in \cref{eq:algmndmor}, the
  functor $F^{*}$ automatically preserves limits and sifted colimits:
  since $U_{S}$ is conservative and preserves limits and sifted
  colimits, it suffices to show that $U_{S}F^{*}$ preserves these, but
  this is by assumption equivalent to $f^{*}U_{T}$, which preserves
  them since both $U_{T}$ and $f^{*}$ do so.
\end{observation}

\begin{observation}\label{rmk:Mdfunctor}
  Given a morphism of Lawvere theories
    \[
    \begin{tikzcd}
      X \arrow{r}{f} \arrow{d}{p} & X' \arrow{d}{p'} \\
      \mathcal{L} \arrow{r}{F} & \mathcal{L}',
    \end{tikzcd}
  \]
  composition with $F$ preserves product-preserving functors, and so
  restricts to a functor $F^{*}\colon \Mod_{\mathcal{L}'}(\mathcal{S})
  \to \Mod_{\mathcal{L}}(\mathcal{S})$. This fits in a commutative
  square
  \[
    \begin{tikzcd}
      \Mod_{\mathcal{L}'}(\mathcal{S}) \arrow{r}{F^{*}} \arrow{d}{p'^{*}} & \Mod_{\mathcal{L}}(\mathcal{S}) \arrow{d}{p^{*}} \\
      \Fun(X', \mathcal{S}) \arrow{r}{f^{*}} & \Fun(X,\mathcal{S}),
    \end{tikzcd}
  \]
  which is a morphism of algebraic monads $T_{\mathcal{L}} \to
  T_{\mathcal{L}'}$ by \cref{propn:ModisAlgMnd}. We thus have a
  functor $\Md \colon \Lawv \to \AlgMnd$ that takes a
  Lawvere theory to its \icat{} of models in $\mathcal{S}$, viewed as
  an algebraic monad via the given map from an \igpd{}.
\end{observation}

\subsection{Lawvere theories from algebraic monads}\label{sec:theoryfrommonad}
Our next goal is to prove that \emph{every} algebraic monad arises
from a Lawvere theory via the following construction:
\begin{defn}
  If $T$ is an algebraic monad on $\Fun(X,\mathcal{S})$, then the
  \emph{Lawvere theory $\Th(T)$ of $T$} is the pair
  $(\mathcal{L}(T), p_{T})$ where $\mathcal{L}(T)^{\op}$ is the full
  subcategory of $\Alg(T)$ spanned by the objects of the form
  $\coprod_{i=1}^{n}F_{T}\Yo_{X} x_{i}$ for $x_{i}$ in $X$, together
  with the map $p_{T} \colon X \to \mathcal{L}(T)$ obtained by restricting
  $F_{T}$.
\end{defn}

\begin{notation}
  It will be convenient to write $(x_{1},\ldots,x_{n})$ for the
  coproduct $\coprod_{i=1}^{n} \Yo_{X}(x)$ in $\Fun(X, \mathcal{S})$,
  and $F_{T}(x_{1},\ldots,x_{n})$ for its image
  $F_{T}(\coprod_{i=1}^{n} \Yo_{X}(x_{i})) \simeq \coprod_{i=1}^{n}
  F_{T}\Yo_{X} x_{i}$ in $\mathcal{L}(T)^{\op}$.
\end{notation}

\begin{lemma}\label{lem:ThTL}
  If $(\mathcal{L}, p)$ is a Lawvere theory, then the Yoneda embedding
  for $\mathcal{L}^{\op}$ factors through a fully faithful functor
  \[ \mathcal{L}^{\op} \hookrightarrow
    \Mod_{\mathcal{L}}(\mathcal{S}) \]
  whose image is precisely $\mathcal{L}(T_{\mathcal{L}})$. This gives
  an equivalence of Lawvere theories
  \[ (\mathcal{L}, p) \simeq \Th(T_{\mathcal{L}}).\]
\end{lemma}
\begin{proof}
  First observe that for $X \in \mathcal{L}$, the copresheaf
  \[ \Yo_{\mathcal{L}^{\op}}X \simeq \Map_{\mathcal{L}}(X, \blank) \]
  clearly preserves products, so that $\Yo_{\mathcal{L}^{\op}}$
  factors through the full subcategory
  $\Mod_{\mathcal{L}}(\mathcal{S})$ of $\Fun(\mathcal{L}, \mathcal{S})$.

  If $L$ denotes the left adjoint to $p^{*} \colon \Mod_{\mathcal{L}}(\mathcal{S}) \to \Fun(X,  \mathcal{S})$, then we saw in
  \cref{propn:ModisAlgMnd} that we have
  \begin{equation}
    \label{eq:LYoX}
   L \Yo_{X}(x) \simeq \Yo_{\mathcal{L}^{\op}}(px).  
  \end{equation}
  Then $L(\coprod_{i=1}^{n} \Yo_{X}(x_{i})) \simeq \coprod_{i=1}^{n}
  \Yo_{\mathcal{L}^{\op}}(px)$, giving natural equivalences
  \[
    \begin{split}
    \Map_{\Mod_{\mathcal{L}}(\mathcal{S})}\left(L\left(\coprod_{i=1}^{n}
  \Yo_{X}(x_{i})\right), \Phi\right) & \simeq \Map_{\Fun(\mathcal{L},
    \mathcal{S})}\left(\coprod_{i=1}^{n} \Yo_{\mathcal{L}^{\op}}(px_{i}), \Phi\right)
  \simeq \prod_{i=1}^{n} \Phi(px_{i}) \\
  & \simeq \Phi\left(\prod_{i=1}^{n}px_{i}\right) \\
  & \simeq
  \Map_{\Mod_{\mathcal{L}}(\mathcal{S})}\left(\Yo_{\mathcal{L}^{\op}}\left(\prod_{i=1}^{n}px_{i}\right),
  \Phi\right)
\end{split}
\]
for any $\Phi \in \Mod_{\mathcal{L}}(\mathcal{S})$. From the Yoneda
Lemma we then have an equivalence
\[ L(\coprod_{i=1}^{n}
  \Yo_{X}(x_{i})) \simeq
  \Yo_{\mathcal{L}^{\op}}(\prod_{i=1}^{n}px_{i}).\]
  By definition, these are precisely the objects that lie in the full
  subcategory $\mathcal{L}(T_{\mathcal{L}})^{\op}$. On the other hand, by assumption
  every object in $\mathcal{L}$ is a product of objects in the image
  of $p$, so these are also precisely the objects in the image of
  $\mathcal{L}^{\op}$. The image of $\mathcal{L}^{\op}$ is thus
  $\mathcal{L}(T_{\mathcal{L}})^{\op}$, as required.

  The equivalence \cref{eq:LYoX} can be rephrased as giving a
  commutative square
  \[
    \begin{tikzcd}
      X \arrow{d}[swap]{p^{\op}} \arrow[hookrightarrow]{r}{\Yo_{X}} &
      \Fun(X, \mathcal{S}) \arrow{d}{L} \\
      \mathcal{L}^{\op}
      \arrow[hookrightarrow]{r}{\Yo_{\mathcal{L}^{\op}}} &
      \Mod_{\mathcal{L}}(\mathcal{S}).
    \end{tikzcd}
  \]
  This restricts to a commutative triangle
  \[
    \begin{tikzcd}
    {} & X \arrow{dl}[swap]{p^{\op}} \arrow{dr}{p_{T}^{\op}} \\
    \mathcal{L}^{\op} \arrow{rr}{\sim} & & \mathcal{L}(T_{\mathcal{L}})^{\op},
  \end{tikzcd}
  \]
  which gives the required equivalence of Lawvere theories.
\end{proof}

\begin{propn}\label{propn:AlgMndisMod} 
  Let $T$ be an algebraic monad on $X$. Then the fully faithful
  inclusion $\mathcal{L}(T)^{\op} \hookrightarrow
  \Alg(T)$ induces a commutative triangle
  \[
    \begin{tikzcd}
      \Alg(T) \arrow{rr}{\nu} \arrow{dr}[swap]{U_{T}} & &
      \Mod_{\mathcal{L}(T)}(\mathcal{S}) \arrow{dl}{p^{*}} \\
       & \Fun(X, \mathcal{S}),
    \end{tikzcd}
  \]
  where $\nu$ arises from the Yoneda embedding of $\Alg(T)$ restricted
  to $\mathcal{L}(T)^{\op}$ and
    $p$ is the induced map $X \to
  \mathcal{L}(T)$. Moreover, the functor $\nu$ is an equivalence.
\end{propn}
\begin{proof}
  By \cref{rmk:Alglocsmall} the \icat{} $\Alg(T)$ is locally small, so
  we have a Yoneda embedding
  $\Alg(T) \hookrightarrow \Fun(\Alg(T)^{\op},
  \mathcal{S})$. Composing this with the inclusion
  $\mathcal{L}(T)^{\op} \hookrightarrow \Alg(T)$, we get a restricted
  Yoneda functor
  \[ \nu \colon \Alg(T) \to \Fun(\mathcal{L}(T),
    \mathcal{S}).\]
  This copresheaf takes $A \in \Alg(T)$ to the copresheaf
  \[ F_{T}(x_{1},\ldots,x_{n}) \in \mathcal{L}(T) \quad \mapsto\quad
    \Map_{\Alg(T)}(F_{T}(x_{1},\ldots,x_{n}),
    A) \simeq \prod_{i=1}^{n} U_{T}A(x_{i}).\]
  This takes finite products in $\mathcal{L}(T)$ to
  products in $\mathcal{S}$, since the former correspond to coproducts
  in $\Alg(T)$, and so factors through the full
  subcategory $\Mod_{\mathcal{L}(T)}(\mathcal{S})$.  Note also that we have a natural
  equivalence
  \begin{equation}
    \label{eq:nuFTeq}
    \nu F_{T}(x_{1},\ldots,x_{n}) \simeq \Yo_{\mathcal{L}(T)^{\op}}
    F_{T}(x_{1},\ldots,x_{n}),
  \end{equation}
  since by \cref{lem:ThTL} we know that
  $\Yo_{\mathcal{L}(T)^{\op}}\Lambda$ lies in
  $\Mod_{\mathcal{L}(T)}(\mathcal{S}) $ for any object 
  $\Lambda \simeq
 F_{T}(y_{1},\ldots,y_{m})$ in $\mathcal{L}(T)$, and so we have
 natural equivalences
  \[
    \begin{split}
    \Map_{\Mod_{\mathcal{L}(T)}(\mathcal{S})}(\Yo_{\mathcal{L}(T)^{\op}}\Lambda,
  \nu F_{T}(x_{1},\ldots,x_{n}))
  & \simeq (\nu F_{T}(x_{1},\ldots,x_{n}))(\Lambda) \\
  & \simeq \Map_{\Alg(T)}(\Lambda,
    F_{T}(x_{1},\ldots,x_{n})) \\
      & \simeq \Map_{\mathcal{L}(T)}(\Lambda, F_{T}(x_{1},\ldots,x_{n})) \\
  & \simeq \Map_{\Mod_{\mathcal{L}(T)}(\mathcal{S})}(\Yo_{\mathcal{L}(T)^{\op}} \Lambda,
   \Yo_{\mathcal{L}(T)^{\op}} F_{T}(x_{1},\ldots,x_{n})).
    \end{split}
  \]

  Now for $A \in \Alg(T)$ we have natural equivalences
  \begin{equation}
    \label{eq:p*nuisU}
   \Map_{\Fun(X,\mathcal{S})}(\Yo_{X}x, p^{*}\nu A) \simeq (\nu A)(px) \simeq
    \Map_{\Alg(T)}(F_{T}x, A) \simeq
    \Map_{\Fun(X,\mathcal{S})}(\Yo_{X}x, U_{T}A),
  \end{equation}
  and so by the Yoneda lemma there is a natural equivalence $p^{*}\nu
  \simeq U_{T}$. This means that we have the required commutative triangle
  \[
    \begin{tikzcd}
      \Alg(T) \arrow{rr}{\nu} \arrow{dr}[swap]{U_{T}} & &
      \Mod_{\mathcal{L}(T)}(\mathcal{S})\arrow{dl}{p^{*}} \\
       & \Fun(X, \mathcal{S}).
    \end{tikzcd}
  \]
  Here both the diagonal functors are monadic right adjoints by
  \cref{propn:ModisAlgMnd}, so to
  prove that the horizontal functor is an equivalence it suffices by
  \cite{HA}*{Corollary 4.7.3.16} to check that the induced transformation
  \[ L \to \nu F_{T} \]
  is an equivalence, where $L$ denotes the left adjoint $\Fun(X,
  \mathcal{S}) \to \Mod_{\mathcal{L}(T)}(\mathcal{S})$ to
  $p^{*}$. Since $p^{*}$
  is conservative, this is equivalent to the transformation of
  endofunctors
  \[ p^{*}L \to p^{*}\nu F_{T} \simeq T \] being an equivalence. Note
  that here $p^{*}$ preserves sifted colimits, since sifted colimits
  in $\Mod_{\mathcal{L}(T)}(\mathcal{S})$ are computed in
  $\Fun(\mathcal{L}(T), \mathcal{S})$, while $T$ preserves sifted
  colimits since it is algebraic. Moreover, every object of
  $\Fun(X, \mathcal{S})$ is a sifted colimit of objects of the form
  $(x_{1},\ldots,x_{n}) := \coprod_{i=1}^{n} \Yo_{X}(x_{i})$ for
  $x_{i} \in X$, so it suffices to check that we have an equivalence for such
  objects.

  We first consider the case where $n =1$. Then for $M \in
  \Mod_{\mathcal{L}(T)}(\mathcal{S})$ we have natural
  equivalences
  \[
    \begin{split}
    \Map_{\Mod_{\mathcal{L}(T)}(\mathcal{S})}(Lx,
    M) & \simeq  \Map_{\Fun(X,\mathcal{S})}(x,
    p^{*}M) \simeq p^{*}M(x) \simeq M(F_{T}x) \\ & \simeq
    \Map_{\Mod_{\mathcal{L}(T)}(\mathcal{S})}(\Yo_{\mathcal{L}(T)^{\op}}
    F_{T}x,  M).       
    \end{split}
\]
  
  Thus by the Yoneda Lemma we have an equivalence $Lx \isoto \Yo_{\mathcal{L}(T)^{\op}}
  F_{T}x$. Since the Yoneda embedding is fully faithful, this
  equivalence is the point of the mapping space $\Map_{\Mod_{\mathcal{L}(T)}(\mathcal{S})}(Lx,
  \Yo_{\mathcal{L}(T)^{\op}}
  F_{T}x)$ that corresponds to $\id_{F_{T}x}$ under the equivalence
  \begin{equation}
    \label{eq:mapLTeq}
    \begin{split}
    \Map_{\Mod_{\mathcal{L}(T)}(\mathcal{S})}(Lx,
    \Yo_{\mathcal{L}(T)^{\op}} F_{T}x) & \simeq  \Map_{\Fun(X,\mathcal{S})}(x,
    p^{*}\Yo_{\mathcal{L}(T)^{\op}} F_{T}x) \\ & \simeq p^{*}\Yo_{\mathcal{L}(T)^{\op}} (F_{T}x)(x) \simeq (\Yo_{\mathcal{L}(T)^{\op}}
    F_{T}x)(F_{T}x) \\ & \simeq \Map_{\Alg(T)}(F_{T}x,
    F_{T}x).
    \end{split}
  \end{equation}    
  On the other hand, the canonical map $L x \to \nu F_{T} x$
  corresponds under the equivalences
  \[ \Map_{\Mod_{\mathcal{L}(T)}(\mathcal{S})}(L x, \nu F_{T} x)
    \simeq \Map_{\Fun(X,\mathcal{S})}(x, p^{*}\nu F_{T}x) 
    \simeq \Map_{\Fun(X,\mathcal{S})}(x, U_{T} F_{T}x) \]
  to the unit map $x \to U_{T}F_{T}x$. 
  For the second equivalence here we used the equivalence $U_{T}
  \simeq p^{*}\nu$, and using
  \cref{eq:p*nuisU} we can decompose this as
  \[ \Map_{\Fun(X,\mathcal{S})}(x, p^{*}\nu F_{T}x) \simeq (p^{*}\nu
    F_{T}x)(x) \simeq \Map_{\Alg(T)}(F_{T}x, F_{T}x)
    \simeq \Map_{\Fun(X,\mathcal{S})}(x, U_{T} F_{T}x),\]
  where the unit map
  corresponds to the identity of $F_{T}x$ under the last adjunction
  equivalence.

  In other words, we have shown that the canonical map $Lx \to \nu
  F_{T}x$ is the point that corresponds to $\id_{F_{T}x}$ under the
  equivalence
  \[ \Map_{\Mod_{\mathcal{L}(T)}(\mathcal{S})}(L x, \nu F_{T} x) \simeq
    \Map_{\Fun(X,\mathcal{S})}(x, p^{*}\nu F_{T}x) 
    \simeq (p^{*}\nu
    F_{T}x)(x) \simeq \Map_{\Alg(T)}(F_{T}x,
    F_{T}x).\]
  It now only remains to observe that when we identify $\nu F_{T}x$ with
  $\Yo F_{T}x$ this equivalence is precisely that of
  \cref{eq:mapLTeq}, which is clear from the definitions. We conclude
  that the composite $L x \to \nu F_{T}x \simeq \Yo F_{T}x$ is
  precisely the map we showed was an equivalence, and so the canonical map
  $L x \to \nu F_{T}x$ is indeed an equivalence.

  Now we consider the general case, where we have a commutative square
  \[
    \begin{tikzcd}
      \coprod_{i=1}^{n} L x_{i} \arrow{r}{\sim} \arrow{d}{\sim} &
      L(x_{1},\ldots,x_{n}) \arrow{d} \\
      \coprod_{i=1}^{n} \nu F_{T}x_{i} \arrow{r} & \nu
      F_{T}(x_{1},\ldots,x_{n}) 
    \end{tikzcd}
  \]
  in which the top horizontal map is an equivalence since $L$
  preserves coproducts and the left vertical map is an equivalence by
  what we just proved. To show the right vertical map is an
  equivalence it is then sufficient to show that the bottom horizontal
  map is one. Mapping this into an object
  $M \in \Mod_{\mathcal{L}(T)}(\mathcal{S})$ we get using the
  identification \cref{eq:nuFTeq} a commutative diagram
    \[
      \begin{tikzcd}
        \Map(\nu F_{T}(x_{1},\ldots,x_{n}), M) \arrow{r}{\sim} \arrow{d} &
        M(F_{T}(x_{1},\ldots,x_{n})) \arrow{d} \\
        \Map(\coprod_{i=1}^{n}\nu F_{T}(x_{i}), M) \arrow{r}{\sim} &
        \prod_{i=1}^{n} M(F_{T}x_{i}).
    \end{tikzcd}
  \]
  Here the right vertical map is an equivalence since $M$ is an
  $\mathcal{L}(T)$-model, hence by the Yoneda lemma the canonical map
  \[ L(x_{1},\ldots,x_{n}) \to \nu F_{T}(x_{1},\ldots,x_{n}) \]
  is an equivalence in $\Mod_{\mathcal{L}(T)}(\mathcal{S})$, as required.
\end{proof}

\begin{remark}
  If $T$ is an algebraic monad, we have seen that $\Alg(T)$ is equivalent to
  $\Mod_{\mathcal{L}(T)}(\mathcal{S})$. In the terminology of
  \cite{HTT}, this identifies $\Alg(T)$ with the
  ``nonabelian derived \icat{}'' of $\mathcal{L}(T)^{\op}$, which by
  \cite{HTT}*{Proposition 5.5.8.15} is the
  \icat{} $\mathcal{P}_{\Sigma}(\mathcal{L}(T)^{\op})$ obtained by
  freely adjoining sifted colimits to $\mathcal{L}(T)^{\op}$.
\end{remark}

\begin{cor}\label{cor:AlgMndMor}\ 
  \begin{enumerate}[(i)]
  \item\label{it:AlgTpres} If $T$ is an algebraic monad on $\Fun(X, \mathcal{S})$, then
    $\Alg(T)$ is a presentable \icat{}.
  \item\label{it:AlgLadj} For any morphism of algebraic monads as in
    \cref{eq:algmndmor}, the functor $F^{*}\colon \Alg(T) \to \Alg(S)$
    has a  left adjoint $F_{!}$.
  \item\label{it:morLawv} The left adjoint $F_{!}$ restricts to a functor
    $\mathcal{L}(S)^{\op} \to \mathcal{L}(T)^{\op}$, which gives a
    morphism of Lawvere theories
    \[
      \begin{tikzcd}
        Y \arrow{d}[swap]{p_{S}} \arrow{r}{f} & X \arrow{d}{p_{T}} \\
        \mathcal{L}(S) \arrow{r}{F} & \mathcal{L}(T).
      \end{tikzcd}
      \]
  \end{enumerate}
\end{cor}
\begin{proof}
  Part \ref{it:AlgTpres} is immediate from \cref{propn:AlgMndisMod}:
  We know that $\Alg(T)$ is equivalent to
  $\Mod_{\mathcal{L}(T)}(\mathcal{S})$, which is presentable by
  \cref{propn:ModisAlgMnd}\ref{it:Modpres}. The functor \[F^{*} \colon
  \Alg(T) \to \Alg(S)\] preserves limits and sifted colimits by
  \cref{rmk:AlgMorpressifted}; it is thus a limit-preserving
  accessible functor between presentable \icats{}, and therefore
  admits a left adjoint $F_{!}$ by \cite{HTT}*{Corollary
    5.5.2.9}. This proves \ref{it:AlgLadj}, and passing to left
  adjoints in the square \cref{eq:algmndmor} we get a commutative
  square
  \[
    \begin{tikzcd}
      \Fun(Y, \mathcal{S}) \arrow{r}{f_{!}} \arrow{d}{F_{S}} & \Fun(X,
      \mathcal{S})
      \arrow{d}{F_{T}}\\
      \Alg(S) \arrow{r}{F_{!}} & \Alg(T).
    \end{tikzcd}
  \]
  Thus we have
  \[ F_{!}F_{S}(\coprod_{i=1}^{n}\Yo_{Y}(y_{i})) \simeq
    F_{T}f_{!}(\coprod_{i=1}^{n}\Yo_{Y}(y_{i})) \simeq
    F_{T}(\coprod_{i=1}^{n} \Yo_{X}f(y_{i})),\]
  which shows that $F_{!}$ takes the full subcategory
  $\mathcal{L}(S)^{\op}$ into $\mathcal{L}(T)^{\op}$, and also
  that it gives the required commutative square. Moreover, $F_{!}$
  preserves coproducts, being a left adjoint, and so the induced
  functor $F \colon \mathcal{L}(S) \to \mathcal{L}(T)$ preserves
  products, since these correspond to coproducts of algebras.
\end{proof}

\begin{observation}
  From \cref{cor:AlgMndMor} we see that the assignment of the Lawvere
  theory $\Th(T)$ to an algebraic monad $T$ extends to a functor
  $\Th \colon \AlgMnd \to \Lawv$. Moreover, this functor is inverse to
  the functor $\Md$ from \cref{rmk:Mdfunctor}: for a Lawvere theory
  $(\mathcal{L},q)$ we have a natural equivalence \[\Th(\Md(\mathcal{L},q))
  \simeq (\mathcal{L},q)\] by \cref{lem:ThTL}, while for an algebraic
  monad $T$ we have a natural equivalence \[\Md(\Th(T)) \simeq T\] by
  \cref{propn:AlgMndisMod}. 
\end{observation}

We have thus proved:
\begin{thm}\label{thm:LawvTheq}
  The functors
  \[ \Md : \Lawv \rightleftarrows \AlgMnd : \Th
  \]
  are inverse equivalences of \icats{}.
\end{thm}

\begin{remark}
  There are many versions of more general ``monad/theory
  correspondences'' for ordinary categories in the
  literature. Versions of these for ``monads with arities'' have been
  generalized to the \icatl{} setting by
  Henry--Meadows~\cite{HenryMeadows} and
  Kositsyn~\cite{KositsynTheories}, though they do not explicitly
  discuss how the situation simplifies in the special case of Lawvere
  theories.
\end{remark}

\section{Analytic monads and $\infty$-operads}\label{sec:anmnds}

We saw in \cref{cor:POpdSpFLT} that pinned $\SpF$-\iopds{} give a full
subcategory of Lawvere theories over $\SpFL$. In particular, we obtain
an algebraic monad from every pinned $\SpF$-\iopd{}, and our overall
goal in this section is to identify this full subcategory with that
coming from a special class of algebraic monads, namely the
\emph{analytic} ones. We start by introducing analytic monads in
\S\ref{sec:analytic-monad}, and then show that any pinned
$\SpF$-\iopd{} gives such an analytic monad in
\S\ref{sec:anmndfromopd}. After this we are ready to study the Lawvere
theories of analytic monads in \S\ref{sec:lawv-theor-analyt}, where we
in particular prove \cref{thm:SpopdLawv} (which together with our
previous work completes the proof of \cref{thm:main}).

\subsection{Analytic and algebraic monads}
\label{sec:analytic-monad}

In this first subsection we will recall the definition of analytic
monads, as introduced in \cite{polynomial}, and give a useful
description of them, as a full subcategory of a certain slice of
algebraic monads.

\begin{defn}
  A functor $F \colon \mathcal{S}_{/X} \to \mathcal{S}_{/Y}$ is
  \emph{analytic} if $F$ preserves sifted colimits and weakly
  contractible limits, that is limits indexed by \icats{} with contractible classifying space.
\end{defn}

\begin{remark}
  By \cite{polynomial}*{Theorem 2.2.3 and Proposition 3.1.9} any
  analytic functor $\mathcal{S}_{/X} \to \mathcal{S}_{/Y}$ is (uniquely) of the form $t_{!}p_{*}s^{*}$ for
  some diagram of spaces
  \[ X \xfrom{s} E \xto{p} B \xto{t} Y,\]
  where the fibres of $p$ are finite sets. This is the
  \emph{polynomial diagram} corresponding to the functor.
\end{remark}

\begin{defn}
  A natural transformation $\alpha \colon F \to G$ of functors
  $\mathcal{C} \to \mathcal{D}$ is \emph{cartesian} if for every
  morphism $\phi \colon x \to y$ in $\mathcal{C}$, the naturality
  square
  \[
    \begin{tikzcd}
      Fx \arrow{r}{\alpha_{x}} \arrow{d}{F\phi} & Gx \arrow{d}{G\phi}
      \\
      Fy \arrow{r}{\alpha_{y}} & Gy
    \end{tikzcd}
  \]
  is cartesian. A \emph{cartesian monad} is a monad whose unit and
  multiplication transformations are cartesian, and an \emph{analytic
    monad} is a cartesian monad on an \icat{} of the form
  $\mathcal{S}_{/X}$ whose underlying endofunctor is analytic.
\end{defn}

\begin{remark}
  An analytic monad can also be defined as a monad in an
  $(\infty,2)$-category whose objects are the \icats{}
  $\Fun(X,\mathcal{S})$ for $X \in \mathcal{S}$, with analytic functors
  as morphisms and cartesian natural transformations as
  2-morphisms. The appropriate notion of morphisms between analytic
  monads for us is then certain lax morphisms of monads in this
  $(\infty,2)$-category. These can also be described in terms of
  monadic adjunctions as follows:
\end{remark}

\begin{defn}
  Suppose $T$ is an analytic monad on $\Fun(X, \mathcal{S})$ and $S$
  is an analytic monad on $\Fun(Y, \mathcal{S})$. A \emph{morphism of
    analytic monads} from $S$ to $T$ is a commutative square of
  \icats{}
  \begin{equation}
    \label{eq:anmndmor}    
    \begin{tikzcd}
      \Alg(T) \arrow{r}{F^{*}} \arrow{d}{U_{T}} & \Alg(S)
      \arrow{d}{U_{S}} \\
      \Fun(X,\mathcal{S}) \arrow{r}{f^{*}} & \Fun(Y, \mathcal{S}),
    \end{tikzcd}
  \end{equation}
  where the bottom horizontal functor comes from a map of spaces $f
  \colon Y \to X$, and the induced transformation
  \[ S f^{*} \to f^{*} T \]
  is cartesian. (Since $U_{S}$ and $U_{T}$ detect pullbacks, we can
  equivalently ask for the mate transformation
  \begin{equation}
    \label{eq:matetr}
   F_{S}f^{*} \to F^{*}F_{T}    
  \end{equation}
  to be cartesian.)
\end{defn}

\begin{defn}
  We define $\AnMnd$ to be the subcategory of $\AlgMnd$ whose objects
  are the analytic monads and whose morphisms are those whose mate
  transformations \cref{eq:matetr} are cartesian.
\end{defn}

The first step towards a more useful description of the \icat{} $\AnMnd$ is to identify its terminal object:
\begin{propn}\label{propn:AnMndterm}
  $\AnMnd$ has a terminal object $\Sym$, given by the unique analytic monad
  structure on the terminal analytic endofunctor on $\mathcal{S}$. The
  latter is the endofunctor
  \[ X \mapsto \coprod_{n=0}^{\infty} X^{\times n}_{h \Sigma_{n}},\]
  given by the diagram
  \begin{equation}
    \label{eq:termanend}
    * \from \xF_{*}^{\simeq} \to \xF^{\simeq} \to *,    
  \end{equation}
  where the map $\xF_{*}^{\simeq} \to \xF^{\simeq}$ is given by forgetting the base point; this is the classifying map for
  morphisms of spaces with finite discrete fibres.\footnote{Note that
    the groupoids $\xF_{*}^{\simeq}$ and $\xF^{\simeq}$ are both
    equivalent to $\coprod_{n=0}^{\infty} B \Sigma_{n}$, and this
    classifying map can be described as the disjoint union of the maps
    $B \Sigma_{n} \to B \Sigma_{n+1}$ induced by the inclusion of
    $\Sigma_{n}$ in $\Sigma_{n+1}$ as the subgroup of automorphisms
    that fix one element in the set $\fset{n+1}$.
  }
\end{propn}
\begin{proof}
  Restricting \cite{polynomial}*{Proposition 4.4.2} to analytic
  monads, we get in particular that the forgetful functor $\AnMnd \to
  \mathcal{S}$ is a cartesian fibration, which lives in a commutative
  triangle
  \[
    \begin{tikzcd}
      \AnMnd \arrow{dr} \arrow{rr}{\mathfrak{U}} &  & \AnEnd \arrow{dl} \\
      & \mathcal{S},
    \end{tikzcd}
  \]
  where $\AnEnd$ is the \icat{} of analytic endofunctors. Here the
  functor $\AnEnd \to \mathcal{S}$ is a cartesian and cocartesian
  fibration, and the forgetful functor $\mathfrak{U} \colon \AnMnd \to \AnEnd$ preserves
  cartesian morphisms. To show that $\AnMnd$ has a terminal object, it
  suffices to prove that the fibre $\AnMnd(*)$ has a terminal object $\Sym$,
  and for every morphism $q \colon X \to *$ in $\mathcal{S}$,
  its cartesian transport $q^{*}\Sym$ is terminal in $\AnMnd(X)$.
    
  The fibre $\AnMnd(X)$ is the \icat{} of associative algebras in the
  \icat{} $\AnEnd(X)$ of analytic endofunctors of $\Fun(X,\mathcal{S})$
  under composition, and the forgetful functor to $\AnEnd(X)$ detects
  limits. Moreover, we have a commutative square
  \[
    \begin{tikzcd}
      \AnMnd(*)\arrow{d}{\mathfrak{U}_{*}} \arrow{r}{q^{*}} &
      \AnMnd(X) \arrow{d}{\mathfrak{U}_{X}} \\
      \AnEnd(*) \arrow{r}{q^{*}} & \AnEnd(*)
    \end{tikzcd}
  \]
  where the lower horizontal functor preserves limits, being a right
  adjoint, and the vertical functors detect limits. Thus if $\Sym$ is
  a terminal object in $\AnMnd(*)$, then $q^{*}\Sym$ is terminal in
  $\AnMnd(X)$ \IFF{} $\mathfrak{U}_{X}q^{*}\Sym \simeq
  q^{*}\mathfrak{U}_{*}\Sym$ is terminal in $\AnEnd(X)$, which is true
  since $q^{*}$ and $\mathfrak{U}_{*}$ preserve limits. Moreover,  
  by \cite{HA}*{Corollary 3.2.2.5} the \icat{} $\AnMnd(*)$
  has a terminal object if $\AnEnd(*)$ has one, and this is
  given by the unique associative algebra structure on this terminal
  analytic endofunctor. To complete the proof we now observe that
  \cite{polynomial}*{Corollary 3.1.13} implies
  that \cref{eq:termanend} gives the terminal object of
  $\AnEnd(*)$, as required.
\end{proof}

Since $\AnMnd$ has a terminal object, the inclusion $\AnMnd \to
\AlgMnd$ factors through $\AlgMnd_{/\Sym}$. We now have:
\begin{lemma}
  The functor $\AnMnd \to \AlgMnd_{/\Sym}$ is fully faithful.
\end{lemma}
\begin{proof}
  Suppose we have analytic monads $T$ and $S$ over spaces $X$ and $Y$, respectively.
  Then a morphism from $S$ to $T$ in $\AlgMnd_{/\Sym}$ gives a
  morphism $f \colon Y \to X$, a natural transformation
  $\alpha \colon Sf^{*} \to f^{*}T$, and a commutative triangle
  \[
    \begin{tikzcd}
      Sp^{*} \arrow{rr}{\alpha q^{*}} \arrow{dr} & & f^{*}Tq^{*} \arrow{dl} \\
       & p^{*}\Sym,
    \end{tikzcd}
  \]
  where $q$ and $p$ denote the maps $X \to *$, $Y \to *$,
  respectively. Here the two diagonal transformations are cartesian,
  since by assumption the two maps to $\Sym$ come from $\AnMnd$, and
  our goal is to prove that $\alpha$ is then cartesian. 
  
  The pasting lemma for pullback squares implies that given a commutative triangle of natural transformations
  \[
    \begin{tikzcd}
      F \arrow{rr}{\phi} \arrow{dr}[swap]{\psi} & & G \arrow{dl}{\chi} \\
       & H
    \end{tikzcd}
  \]
  where $\chi$ is cartesian, the transformation $\phi$ is cartesian
  \IFF{} $\psi$ is cartesian. From the triangle above we can thus
  conclude that $\alpha q^{*}$ is cartesian.
  We can then make a commutative square
  \[
    \begin{tikzcd}
      Sf^{*} \arrow{r}{\alpha} \arrow{d}  & f^{*}T \arrow{d} \\
      Sp^{*}q_{!} \arrow{r}{\alpha q^{*}q_{!}} & f^{*}Tq^{*}q_{!}
    \end{tikzcd}
  \]
  using the counit transformation $\id \to q^{*}q_{!}$, which is
  cartesian by \cite{polynomial}*{Lemma 2.1.5}. Then the vertical maps
  in the square are cartesian transformations, as is the bottom
  horizontal map. Using the pasting lemma again, it follows that the
  top horizontal transformation $\alpha$ is also cartesian.
\end{proof}

\begin{observation}\label{obs:ThAnMnd}
  From \cref{thm:LawvTheq} it now follows that the functor
  \[ \AnMnd \hookrightarrow \AlgMnd_{/\Sym} \xto{\Th}
    \Lawv_{/\Th(\Sym)} \]
  identifies $\AnMnd$ with a full subcategory
  $\LawvAnpre$ of $\Lawv_{/\Th(\Sym)}$. Over the next few sections we
  will identify this full subcategory with that of pinned $\SpF$-\iopds{}.
\end{observation}

\begin{remark}
  The \icat{} $\AnMnd$ can also be seen as a (non-full)
  subcategory of $\AlgMnd$, which means that it also corresponds to
  some subcategory of $\Lawv$. It would be interesting to know if this
  has an explicit description. In the 1-categorical setting, Szawiel
  and Zawadowski give such a description of the Lawvere theories of one-sorted
  analytic monads in \cite{SzZ}.
\end{remark}

\subsection{Analytic monads from $\SpF$-\iopds{}}\label{sec:anmndfromopd}
In this subsection we will describe the monad corresponding to a pinned
$\SpF$-\iopd{}, and show that this is always analytic. Using this, we
will then identify the \icat{} $\POpd(\SpF)$ with a full subcategory of
$\AnMnd$.

\begin{notation}
  Suppose $\mathcal{O}$ is a $\SpF$-\iopd{}. We write
  $U_{\mathcal{O}}$ for the restriction functor
  \[\Mon_{\mathcal{O}}(\mathcal{S}) \hookrightarrow \Fun(\mathcal{O},
    \mathcal{S}) \to \Fun(\mathcal{O}_{\fset{1}}^{\simeq},
    \mathcal{S})\] and $F_{\mathcal{O}}$ for its left adjoint. Note
  that $(\mathcal{O}, \mathcal{O}_{\fset{1}}^{\simeq})$ is a Lawvere
  theory by \cref{lem:SpFopdprod}, and by \cref{propn:SpFopdprod} the
  \icat{} $\Mon_{\mathcal{O}}(\mathcal{S})$ is the \icat{} of models
  for this Lawvere theory. Thus $U_{\mathcal{O}}$ is the monadic right
  adjoint for an algebraic monad by \cref{propn:ModisAlgMnd}; we write
  $T_{\mathcal{O}}$ for this monad.
\end{notation}

\begin{propn}\label{propn:opdmonad}
  Suppose $\mathcal{O}$ is an $\SpF$-\iopd{}. Then $T_{\mathcal{O}}$
  is an analytic monad, and its underlying endofunctor is given on
  $F \in \Fun(\mathcal{O}_{\fset{1}}^{\simeq}, \mathcal{S})$ by the
  formula
  \[ (T_{\mathcal{O}}F)(x) \simeq \colim_{y \actto x \in
      \Act_{\mathcal{O}}(x)}  \prod_{y \intto y_{i}}F(y_{i}),\]
  where $\Act_{\mathcal{O}}(x) := (\mathcal{O}^{\act}_{/x})^{\simeq}$
  denotes the \igpd{} of active maps to $x$ in
  $\mathcal{O}$ and the limit is over the set of inert maps $y \intto y_{i}$
  over $\rho'_{i}$.  Moreover, any morphism of $\SpF$-\iopds{} induces a
  morphism of analytic monads.
\end{propn}
\begin{proof}
  In this proof it is convenient to use some terminology from \cite{patterns1}. The
  \icat{} $\mathcal{O}$, equipped with its inert-active factorization
  system and the objects over $\fset{1}$, is an \emph{algebraic
    pattern}, and $\Mon_{\mathcal{O}}(\mathcal{S})$ is the \icat{} of
  Segal $\mathcal{O}$-spaces for this pattern.

  The algebraic pattern $\SpF$ is \emph{extendable} by
  \cite{BHS}*{Example 3.3.25}, hence any $\SpF$-\iopd{} is also
  extendable by \cite{patterns1}*{Corollary 9.17} (or
  \cite{BHS}*{Lemma 4.1.15}). The formula for $T_{\mathcal{O}}$
  is then a special case of \cite{patterns1}*{Corollary 8.12}.

  It then follows from \cite{patterns1}*{Proposition 10.6} that the monad
  $T_{\mathcal{O}}$ is cartesian and preserves weakly contractible
  limits. Since we already know that $T_{\mathcal{O}}$ is algebraic,
  \ie{} preserves sifted colimits, this shows that $T_{\mathcal{O}}$
  is an analytic monad. A morphism of $\SpF$-\iopds{} then induces a
  morphism of analytic monads by \cite{patterns1}*{Proposition
    10.10}.
\end{proof}

\begin{cor}
  Suppose $(\mathcal{O}, q\colon X \twoheadrightarrow
  \mathcal{O}^{\simeq}_{\fset{1}})$ is a pinned $\SpF$-\iopd{}. Then
  the composite
  \[ \Mon_{\mathcal{O}}(\mathcal{S}) \xto{U_{\mathcal{O}}}
    \Fun(\mathcal{O}_{\fset{1}}^{\simeq}, \mathcal{S}) \xto{q^{*}}
    \Fun(X, \mathcal{S})\]
  is the monadic right adjoint for an analytic monad
  $T_{(\mathcal{O},q)}$. Moreover, any
  morphism of pinned $\SpF$-\iopds{} induces a morphism of analytic monads.
\end{cor}

\begin{proof}
  The endofunctor $T_{(\mathcal{O},q)}$ is the composite
  $q^{*}T_{\mathcal{O}}q_{!}$. Here the functor $q^{*}$ preserves all
  limits and colimits, and $q_{!}$ preserves all colimits (being a
  left adjoint) as well as weakly contractible limits (by
  \cite{polynomial}*{Lemma 2.2.10}), hence $T_{(\mathcal{O},q)}$ is an
  analytic functor since $T_{\mathcal{O}}$ is one. Moreover, the
  multiplication and unit transformations for $T_{(\mathcal{O},q)}$
  are given by the compositions
  \[ q^{*}T_{\mathcal{O}}q_{!}q^{*}T_{\mathcal{O}}q_{!} \to
    q^{*}T_{\mathcal{O}}T_{\mathcal{O}}q_{!} \to q^{*}T_{\mathcal{O}}q_{!},\qquad
    \id \to q^{*}q_{!} \to q^{*}T_{\mathcal{O}}q_{!}\]
  of the corresponding transformations for $T_{\mathcal{O}}$ and the
  counit and unit of the adjunction $q_{!} \dashv q^{*}$. The latter
  are cartesian transformations by \cite{polynomial}*{Lemma 2.1.5},
  hence so are these compositions, since we know $T_{\mathcal{O}}$ is
  a cartesian monad. Thus $T_{(\mathcal{O},q)}$ is indeed an analytic monad.
  Similarly, the natural transformation associated
  to a morphism of pinned $\SpF$-\iopds{} is built from that coming from the
  underlying morphism of $\SpF$-\iopds{} and cartesian (co)unit
  transformations, so that we get a morphism of analytic monads.
\end{proof}

The monad $T_{(\mathcal{O},q)}$ for a pinned $\SpF$-\iopd{}
$(\mathcal{O},q)$ is by definition the monad corresponding to the
Lawvere theory $(\mathcal{O},q)$. We next observe that the monad
corresponding to the terminal (pinned) $\SpF$-\iopd{} is the terminal
analytic monad:
\begin{lemma}
  The monad $T_{\SpFL}$ corresponding to the Lawvere theory $\SpFL$ is
  $\Sym$. Equivalently, the Lawvere theory $\Th(\Sym)$ is $\SpFL$.
\end{lemma}
\begin{proof}
  We know from \cref{propn:AnMndterm} that $\Sym$ is the unique analytic monad structure on its
  underlying endofunctor on $\mathcal{S}$, which is
  \[ X \mapsto \coprod_{n=0}^{\infty} X^{\times n}_{h \Sigma_{n}}.\]
  It thus suffices to show that this is also the underlying
  endofunctor of the analytic monad $T_{\SpF}$. This follows from the
  formula in \cref{propn:opdmonad} since we have
  $\Act_{\SpF}(\fset{1}) \simeq \coprod_{n=0}^{\infty} B\Sigma_{n}$,
  giving
  \[ T_{\SpF}(X) \simeq \colim_{\fset{n} \actto \fset{1} \in
      \Act_{\SpF}(\fset{1})} X^{\times n} \simeq
    \coprod_{n=0}^{\infty} X^{\times n}_{h \Sigma_{n}}.\]
\end{proof}

We have thus shown that both $\POpd(\SpF)$ and $\AnMnd$ give full
subcategories of $\LawvSpF$, and that the former subcategory is
contained in the latter:
\begin{cor}\label{cor:POpdinLawvAn}
  $\POpd(\SpF)$ is a full subcategory of $\LawvAn \simeq \AnMnd$. \qed
\end{cor}

\subsection{Lawvere theories of analytic monads}
\label{sec:lawv-theor-analyt}

In the previous subsection we showed that we have a fully faithful functor
$\POpd(\SpF) \hookrightarrow \LawvAn$. Our goal in this subsection is to
prove that this is also essentially surjective. In other words, we
want to show that the Lawvere theory of \emph{any} analytic monad is a
pinned $\SpF$-\iopd{}. To show this, we will first describe an
inert--active factorization system on such Lawvere theories. This
arises from a factorization system on the entire Kleisli \icat{},
which was constructed in \cite{patterns1}. We begin by recalling this,
which requires first introducing some terminology:

\begin{defn}
  A functor $F \colon \mathcal{C} \to \mathcal{D}$ is a \emph{local
    right adjoint} if for every object $C \in \mathcal{C}$ the induced
  functor on slices
  $F_{/C} \colon \mathcal{C}_{/C} \to \mathcal{D}_{/FC}$ is a right
  adjoint. 
\end{defn}

\begin{observation}
  By \cite{polynomial}*{Theorem 2.2.3}, a functor $F \colon \Fun(X,
  \mathcal{S}) \to \Fun(Y, \mathcal{S})$ is a local right adjoint
  \IFF{} it is accessible and preserves weakly contractible limits. In
  particular, every analytic functor is a local right adjoint.
\end{observation}

\begin{observation}\label{rmk:anmndL*}
  For an analytic functor
  $F \colon \Fun(X, \mathcal{S}) \to \Fun(Y, \mathcal{S})$ we can
  describe the left adjoint $L_{*}$ of $F_{/*}$ quite explicitly:
  Recall that an analytic functor $F$ is of the form $t_{!}p_{*}s^{*}$
  for a unique bispan of spaces
  \[ X \xfrom{s} E \xto{p} B \xto{t} Y \]
  where $p$ has finite discrete fibres.
  Here $t \colon B \to Y$ corresponds to $F(*)$ under the
  straightening equivalence $\Fun(Y, \mathcal{S}) \simeq \mathcal{S}_{/Y}$,
  since $t_{!}$ is given on slices by composition with $t$ and $p_{*}s^{*}$
  preserves the terminal object. Thus $F_{/*}$ can be identified with
  the functor $p_{*}s^{*} \colon \Fun(X, \mathcal{S})  \to
  \Fun(B, \mathcal{S})$, with the left adjoint $L_{*}$ being 
   $s_{!}p^{*}$. In
  particular, a map $\Yo_{Y}(y) \to F(*)$ corresponds to a point $b \in B$
  over $y \in Y$, and $L_{*}\Yo_{Y}(y)$ corresponds to $s_{!}p^{*}b$, which is
  the composite $E_{b} \hookrightarrow E \xto{s} X$. Since $E_{b}$ is
  a finite set, we have
  \begin{equation}
    \label{eq:L*Yo}
    L_{*}(\Yo_{Y}(y)\to F(*)) \simeq \coprod_{e \in E_{b}} \Yo(s(e)).
  \end{equation}
\end{observation}

\begin{defn}
  For any monad $T$ on some \icat{} $\mathcal{C}$ , we write $\mathcal{K}(T)$ for the \emph{Kleisli
    \icat{}} of $T$, defined as the full subcategory of $\Alg(T)$
  containing the free algebras $F_{T}C$ for $C \in \mathcal{C}$.
\end{defn}

\begin{defn}
  Suppose $T$ is an analytic monad on $\Fun(X, \mathcal{S})$, and
  consider a morphism $\phi \colon F_{T}I \to F_{T}J$ in
  $\mathcal{K}(T)$. Under the adjunction $F_{T} \dashv U_{T}$ this
  corresponds to a morphism $\phi' \colon I \to TJ$ in $\Fun(X,
  \mathcal{S})$, which we can regard as a morphism in $\Fun(X,
  \mathcal{S})_{/T*}$ via the image $TJ \to T*$ under $T$ of the
  unique map from $J$ to the terminal object. Then $\phi'$ corresponds
  to a morphism $\phi'' \colon L_{*}I \to J$ in $\Fun(X,
  \mathcal{S})$, where $L_{*}$ is the left adjoint to $T_{/*}$. The
  unit of the
  adjunction $L_{*} \dashv T_{/*}$ then gives a factorization of
  $\phi'$ as
  \[ J \to TL_{*}J \xto{T \phi''} TI,\]
  which is adjoint to a factorization of $\phi$ as
  \[ F_{T}J \to F_{T}L_{*}J \xto{F_{T}\phi''} F_{T}I.\]
  We call this the \emph{canonical factorization} of $\phi$, and say
  that $\phi$ is \emph{inert} if the map $F_{T}J \to F_{T}L_{*}J$ is
  an equivalence, and \emph{active} if $F_{T}\phi''$ is an equivalence.
\end{defn}

\begin{thm}\label{thm:Kleislifact}
  If $T$ is an analytic monad, then the active and inert morphisms
  form a factorization system on $\mathcal{K}(T)$, and the canonical
  factorization is an active--inert factorization. Moreover, for any
  morphism of analytic monads \cref{eq:anmndmor}, the functor $F_{!}
  \colon \mathcal{K}(S) \to \mathcal{K}(T)$ preserves active and inert
  morphisms.
\end{thm}
\begin{proof}
  That active and inert morphisms form a factorization system is a
  special case of \cite{patterns1}*{Theorem 12.1}, while their
  preservation by morphisms of analytic monads follows from the same
  argument as in the proof of \cite{patterns1}*{Lemma 11.18}.
\end{proof}

\begin{cor}\label{cor:LTfact}
  Let $T$ be an analytic monad on $\Fun(X, \mathcal{S})$. The inert--active factorization system on $\mathcal{K}(T)^{\op}$ restricts
  to the Lawvere theory $\mathcal{L}(T)$. Moreover, for any morphism
  of analytic monads \cref{eq:anmndmor}, the functor $F \colon
  \mathcal{L}(S) \to \mathcal{L}(T)$ preserves inert and active moprhisms.
\end{cor}
\begin{proof}
  Given a morphism $F_{T}(\coprod_{i=1}^{n}\Yo_{X}(x_{i})) \to
  F_{T}(\coprod_{j=1}^{m} \Yo_{X}(y_{j}))$, its canonical factorization is
  \[ F_{T}(\coprod_{i=1}^{n}\Yo_{X}(x_{i})) \to
    F_{T}L_{*}(\coprod_{i=1}^{n}\Yo_{X}(x_{i})) \to
    F_{T}(\coprod_{j=1}^{m} \Yo_{X}(y_{j}))\]
  where $L_{*}(\coprod_{i=1}^{n}\Yo_{X}(x_{i}))$ is defined with respect
  to the composite \[\coprod_{i=1}^{n}\Yo_{X}(x_{i})  \to
  T(\coprod_{j=1}^{m} \Yo_{X}(y_{j})) \to T(*)\] of the adjoint map and $T$
  applied to the unique map $\coprod_{j=1}^{m} \Yo_{X}(y_{j}) \to *$. It
  thus suffices to check that $L_{*}(\coprod_{i=1}^{n}\Yo_{X}(x_{i}))$ is
  again a finite coproduct of representables for any map
  $\coprod_{i=1}^{n}\Yo_{X}(x_{i}) \to T(*)$. This follows from
  \cref{eq:L*Yo} in the case $n=1$ and so in general since we have
  $L_{*}(\coprod_{i=1}^{n}\Yo_{X}(x_{i}))
  \simeq \coprod_{i=1}^{n} L_{*}\Yo_{X}(x_{i})$ as $L_{*}$ is a left
  adjoint. The preservation of active and inert morphisms then follows
  from \cref{thm:Kleislifact}.
\end{proof}

\begin{propn}\label{propn:anmndmorcart}
  Consider a morphism of analytic monads \cref{eq:anmndmor}, and let
  $F_{!}$ be the left adjoint to $F^{*}$ (which exists by
  \cref{cor:AlgMndMor}). Then:
  \begin{enumerate}[(i)]
  \item Every inert morphism in $\mathcal{K}(S)$ is
    $F_{!}$-cartesian.
  \item Given $F_{S}I \in \mathcal{K}(S)$ and a morphism $\psi \colon J \to
    f_{!}I$ in $\Fun(X, \mathcal{S})$, there exists an
    $F_{!}$-cartesian morphism over $F_{T}\psi$ in
    $\mathcal{K}(T)$, given by $F_{S}\phi$ where $\phi$ is determined by
    the pullback square
    \[
      \begin{tikzcd}
        K \arrow{r}{\phi} \arrow{d} & I \arrow{d} \\
        f^{*}J \arrow{r}{f^{*}\psi} & f^{*}f_{!}I
      \end{tikzcd}
    \]
    in $\Fun(Y, \mathcal{S})$.
  \item A morphism in $\mathcal{K}(S)$ is inert \IFF{} it is cartesian
    over an inert morphism in $\mathcal{K}(T)$.
  \end{enumerate}
\end{propn}
\begin{proof}
  Every inert morphism in $\mathcal{K}(S)$ is a composite of a free
  morphism and an equivalence. Since equivalences are always
  cartesian, to prove (i) it suffices to show that free morphisms in
  $\mathcal{K}(S)$ are cartesian. Given a morphism $\phi \colon I \to
  J$ in $\Fun(Y, \mathcal{S})$, we must show that for every object $K \in
  \Fun(Y, \mathcal{S})$ the commutative square
  \[
    \begin{tikzcd}
      \Map_{\Alg(S)}(F_{S}K, F_{S}I) \arrow{r}{(F_{S}\phi)_{*}} \arrow{d} & \Map_{\Alg(S)}(F_{S}K,
      F_{S}J) \arrow{d} \\
      \Map_{\Alg(T)}(F_{!}F_{S}K, F_{!}F_{S}I)
      \arrow{r}{(F_{!}F_{S}\phi)_{*}} &
      \Map_{\Alg(T)}(F_{!}F_{S}K, F_{!}F_{S}J)
    \end{tikzcd}
  \]
  is cartesian. Using the natural equivalence $F_{!}F_{S}
  \simeq F_{T}f_{!}$ and the various adjunctions involved, we identify
  this with the commutative square
    \[
    \begin{tikzcd}
      \Map_{\Fun(Y, \mathcal{S})}(K, SI) \arrow{r}{(S\phi)_{*}}
      \arrow{d} & \Map_{\Fun(Y, \mathcal{S})}(K, SJ) \arrow{d} \\
      \Map_{\Fun(Y, \mathcal{S})}(K, f^{*}Tf_{!}I)
      \arrow{r}{(f^{*}Tf_{!}\phi)_{*}} &
      \Map_{\Fun(Y, \mathcal{S})}(K, f^{*}Tf_{!}J).
    \end{tikzcd}
  \]
  This is cartesian for all $K$ \IFF{} the commutative square
  \[
    \begin{tikzcd}
      SI \arrow{r}{S\phi} \arrow{d} & SJ \arrow{d} \\
      f^{*}Tf_{!}I \arrow{r}{f^{*}Tf_{!}\phi} & f^{*}Tf_{!}J
    \end{tikzcd}
  \]
  is a cartesian square in $\Fun(Y, \mathcal{S})$. Now this square is
  indeed cartesian since it is a naturality square for the natural
  transformation $S \to Sf^{*}f_{!} \to f^{*}Tf_{!}$ which is
  cartesian since all the functors involved preserve pullbacks and the
  unit transformation $\id \to f^{*}f_{!}$ and the transformation
  $Sf^{*} \to f^{*}T$ are both cartesian. This proves (i).

  To prove (ii), first define $\phi \colon K \to I$ by the pullback
  square
    \[
      \begin{tikzcd}
        K \arrow{r}{\phi} \arrow{d} & I \arrow{d} \\
        f^{*}J \arrow{r}{f^{*}\psi} & f^{*}f_{!}I.
      \end{tikzcd}
    \]
  Then consider the commutative diagram
  \[
    \begin{tikzcd}
      f_{!}K \arrow{r}{f_{!}\phi} \arrow{d} & f_{!}I \arrow{d}
      \arrow[bend left=40,equals]{dd} \\
      f_{!}f^{*}J \arrow{r}{f_{!}f^{*}\psi}  \arrow{d} &
      f_{!}f^{*}f_{!}I  \arrow{d}  \\
      J \arrow{r}{\psi} & f_{!}I
    \end{tikzcd}
  \]
  in $\Fun(X, \mathcal{S})$. Here the bottom square is cartesian since the
  counit $f_{!}f^{*} \to \id$ is cartesian, and the top square is
  cartesian since
  $f_{!}$ preserves pullbacks. The composite square is then also
  cartesian, and since the right vertical composite is the identity,
  it follows that the map $f_{!}K \to J$ is an equivalence. The
  composite square thus gives an equivalence between $f_{!}\phi$ and
  $\psi$. The free morphism $F_{S}\phi$ is then cartesian over
  $F_{T}f_{!}\phi \simeq F_{T}\psi$ by part (i), which proves (ii).

  To prove (iii), it remains to prove that every cartesian morphism in
  $\mathcal{K}(S)$ that lies over an inert morphism in
  $\mathcal{K}(T)$ is inert. Since inert morphisms are composites of
  free morphisms and equivalences, we may assume the morphism in
  $\mathcal{K}(T)$ is free. Then the construction in (ii) gives a
  cartesian morphism with the same image in $\mathcal{K}(T)$ that is
  manifestly inert. Since cartesian morphisms are unique when they
  exist, this completes the proof.
\end{proof}

Restricting this to Lawvere theories, we have:
\begin{cor}\label{cor:LSintact}
  Consider a morphism of analytic monads \cref{eq:anmndmor}, and let
  $F \colon \mathcal{L}(S) \to \mathcal{L}(T)$ be the induced morphism
  of Lawvere theories. Then:
  \begin{enumerate}[(i)]
  \item   $\mathcal{L}(S)$ has $F$-cocartesian
    morphisms over inert morphisms in $\mathcal{L}(T)$, and these are
    precisely the inert morphisms in $\mathcal{L}(S)$.  
  \item A morphism in $\mathcal{L}(S)$ is active \IFF{} it lies over an
    active morphism in $\mathcal{L}(T)$.
  \end{enumerate}
\end{cor}
\begin{proof}
  By \cref{propn:anmndmorcart} every inert morphism in
  $\mathcal{K}(S)$ is cartesian over $\mathcal{K}(T)$. Restricting
  to the full subcategories $\mathcal{L}(S)^{\op}$ and
  $\mathcal{L}(T)^{\op}$, we get after taking op that every inert
  morphism in $\mathcal{L}(S)$ is cocartesian over
  $\mathcal{L}(T)$. To show that $\mathcal{L}(S)$ has
  cocartesian morphisms over inert maps in $\mathcal{L}(T)$, it is
  enough to
  check that given $\Xi := \coprod_{i=1}^{n} \Yo_{Y}(y_{i})$ and a
  morphism $\psi \colon \coprod_{j=1}^{m}\Yo_{X}(x_{j}) \to f_{!}\Xi
  \simeq \coprod_{i=1}^{n}\Yo_{X}f(y_{i})$ in
  $\Fun(X,\mathcal{S})$, the cartesian morphism over $\psi$ from
  \cref{propn:anmndmorcart}(ii) lies in the full subcategory
  $\mathcal{L}(S)^{\op}$. First observe that the map $\psi$ amounts to
  specifiying a morphism $h \colon \fset{m} \to \fset{n}$ and
  equivalences $x_{j} \simeq f(y_{h(j)})$ in $X$. In the cartesian
  square
  \[
    \begin{tikzcd}
      K \arrow{r} \arrow{d} &
      \coprod_{i=1}^{n} \Yo(y_{i}) \arrow{d} \\
      f^{*}\coprod_{j=1}^{m}\Yo_{X}(x_{j}) \arrow{r} &
      f^{*}\coprod_{i=1}^{n}\Yo_{X}f(y_{i}),
    \end{tikzcd}
  \]
  that induces the cartesian morphism in $\mathcal{K}(S)$, the object
  $K$ can therefore be identified with $\coprod_{j=1}^{m}
  \Yo(y_{h(j)})$, using that coproducts are disjoint
  and colimits are universal in $\Fun(Y, \mathcal{S})$. To prove (i)
  it remains to show that
  \emph{every} morphism in $\mathcal{L}(S)$ that is cocartesian over
  an inert map in $\mathcal{L}(T)$ is inert; this is true since cocartesian
  morphisms are unique when they exist and the construction in
  \cref{propn:anmndmorcart}(ii) gives a cocartesian morphism that
  \emph{is} inert.
  
  Since $\mathcal{L}(S)$ has cocartesian morphisms over inert maps in
  $\mathcal{L}(T)$, we can lift the inert--active factorization system
  on $\mathcal{L}(T)$
  to a factorization system on $\mathcal{L}(S)$, whereby a morphism
  factors as a cocartesian morphism over an inert map in $\mathcal{L}(T)$
  followed by a morphism that maps to an active map in
  $\mathcal{L}(T)$. By (i),
  the first class is precisely that of inert morphisms in
  $\mathcal{L}(S)$. But one class in a factorization system is completely
  determined by the other by \cite{HTT}*{Proposition 5.2.8.11}, so
  this means that the active morphisms in
  $\mathcal{L}(S)$ must be
  exactly those that lie over active morphisms in $\mathcal{L}(T)$.
\end{proof}

\begin{lemma}\label{lem:SpFinertagree}
  Under the identification $\SpF \simeq \mathcal{L}(\Sym)$, the
  inert--active factorization system on $\SpF$ (as in
  \cref{defn:SpFactint}) corresponds to that from \cref{cor:LTfact}.
\end{lemma}
\begin{proof}
  We will show that both $\SpF$ and $\mathcal{L}(\Sym)$ can be
  identified with the canonical pattern $\mathcal{W}(\Sym)$ associated to the monad
  $\Sym$ in \cite{patterns1}*{\S 13}. There $\mathcal{W}(\Sym)^{\op}$
  is defined as the full subcategory of $\Alg(\Sym) \simeq
  \Mon_{\SpF}(\mathcal{S})$ of free algebras on objects on the full
  subcategory $\mathcal{U}(\Sym)^{\op} \subseteq \mathcal{S}$
  consisting of objects of the form $L_{*}(p)$ for morphisms $p \colon
  * \to \Sym(*)$. By \cref{eq:L*Yo} these objects are precisely the
  finite coproducts of the point, \ie{} the finite sets, so that
  $\mathcal{W}(\Sym)^{\op}$ is precisely
  $\mathcal{L}(\Sym)^{\op}$. The factorization system on both is
  obtained by restricting that on $\mathcal{K}(\Sym)$, and both have
  the free algebra on a point as their unique elementary object.

  It remains to show that the canonical pattern $\mathcal{W}(\Sym)$ is
  in fact $\SpF$. For this we apply \cite{patterns1}*{Corollary
    14.18}. As already observed, the pattern $\SpF$ is
  \emph{extendable} by \cite{BHS}*{Example 3.3.25}, and by
  inspection every object of $\SpF$ admits a (unique) active morphism
  to the elementary object $\fset{1}$. The only condition left to
  check is then that the pattern $\SpF$ is \emph{saturated}, meaning
  that for every object $S \in \SpF$, the functor
  $\Map_{\SpF}(S,\blank) \colon \SpF \to \mathcal{S}$ is an
  $\SpF$-monoid. This amounts to the functor taking disjoint unions of
  sets to products, or in other words that the disjoint union is the
  cartesian product in $\SpF$, which we saw in \cref{lem:SpFprod}.
\end{proof}

\begin{propn}\label{propn:LTopd}
  Let $T$ be an analytic monad on $\Fun(X, \mathcal{S})$. Then the
  morphism of Lawvere theories $\mathcal{L}(T) \to \SpF$, corresponding
  to the unique morphism of analytic monads
  \[
    \begin{tikzcd}
      \Alg(\Sym) \arrow{d}{U_{\Sym}} \arrow{r}{Q^{*}} & \Alg(T)
      \arrow{d}{U_{T}} \\
      \mathcal{S} \arrow{r}{q^{*}} & \Fun(X, \mathcal{S})
    \end{tikzcd}
    \]
    to $\Sym$, is a $\SpF$-\iopd{}. Moreover, the functor
    $p \colon X \to \mathcal{L}(T)$ makes
    $\Th(T) = (\mathcal{L}(T),p)$ a pinned $\SpF$-\iopd{}.
\end{propn}
\begin{proof}
  We check the conditions \ref{spfwsf1}--\ref{spfwsf3} in
  \cref{SpFopdcond}. Condition \ref{spfwsf1} follows from 
  \cref{cor:LSintact} and the identification of the inert morphisms in
  $\SpF$ from \cref{lem:SpFinertagree}.
  The pullback square
  \[
    \begin{tikzcd}
      F_{T}x_{i} \arrow{r} \arrow{d} & F_{T}(x_{1},\ldots,x_{n})
      \arrow{d} \\
      F_{T}q^{*}\{i\} \arrow{r} & F_{T}q^{*}\fset{n}
    \end{tikzcd}
  \]
  in $\mathcal{K}(T)$ shows, by the description of the cartesian
  morphisms in  \cref{propn:anmndmorcart}, that the inclusion of the
  factor $F_{T}x_{i}$ in the coproduct
  $F_{T}(x_{1},\ldots,x_{n})\simeq \coprod_{i=1}^{n}F_{T}x_{i}$ is
  cartesian over the inclusion $\{i\} \hookrightarrow \fset{n}$. Over
  in $\mathcal{L}(T)$, this
  means that the cocartesian morphisms $F_{T}(x_{1},\ldots,x_{n}) \to
  F_{T}(x_{i})$ over $\rho_{i}$ exhibit $F_{T}(x_{1},\ldots, x_{n})$
  as the product $\prod_{i=1}^{n}F_{T}x_{i}$, which is precisely
  condition \ref{spfwsf2}. Moreover, given objects $F_{T}x_{i}$ for $i
  = 1,\ldots,n$, the same argument shows that the object
  $F_{T}(x_{1},\ldots,x_{n})$ in $\mathcal{L}(T)_{\fset{n}}$ is sent
  to $(F_{T}(x_{1}),\ldots, F_{T}(x_{n}))$ in $\prod_{i=1}^{n}
  \mathcal{L}(T)_{\fset{1}}$ under cocartesian transport along the
  maps $\rho_{i}$. This gives condition \ref{spfwsf3}, which completes
  the proof that $\mathcal{L}(T)$ is a $\SpF$-\iopd{}.

  We have a commutative square
  \[
    \begin{tikzcd}
      X \arrow{r}{p} \arrow{d} & \mathcal{L}(T) \arrow{d} \\
      \{\fset{1}\} \arrow{r} & \SpF,
    \end{tikzcd}
  \]
  which shows that $p$ factors through a morphism $X \to
  \mathcal{L}(T)_{\fset{1}}^{\simeq}$. It remains to show that this is
  essentially surjective. Every object of $\mathcal{L}(T)$ is a
  product of objects in the image of $p$, but since the functor to
  $\SpF$ preserves products, these products cannot lie over $\fset{1}$
  if they have more than a single factor. Thus the only objects of
  $\mathcal{L}(T)$ that map to $\fset{1}$ are those in the image of
  $p$, as required.
\end{proof}

\cref{propn:LTopd} shows that the Lawvere theory of any analytic monad
is a pinned $\SpF$-\iopd{}. We already saw in \cref{cor:POpdinLawvAn}
that $\POpd(\SpF)$ was a full subcategory of $\LawvAn$, so this means
this subcategory is actually the entire \icat{} $\LawvAn$. Combining this with \cref{thm:LawvTheq} and
\cref{obs:ThAnMnd}, we get:
\begin{cor}\label{cor:opdmndeq}
  We have mutually inverse equivalences
  \[ \Md : \POpd(\SpF) \rightleftarrows \AnMnd : \Th \]
  induced by those of \cref{thm:LawvTheq}.
\end{cor}

\section{$\infty$-operads and complete dendroidal Segal spaces}\label{sec:completedend}

Our work so far gives an equivalence
\[ \POpd(\SpF) \simeq \AnMnd.\]
On the other hand, in \cite{polynomial} we identified $\AnMnd$ with
the \icat{} $\Seg_{\Oop}(\mathcal{S})$ of \emph{dendroidal Segal
  spaces}. Our goal in this section is to prove \cref{thm:csegO},
\ie{} to show that the resulting
composite equivalence between $\POpd(\SpF)$ and
$\Seg_{\Oop}(\mathcal{S})$ restricts to an equivalence between
$\Opd(\SpF)$, which we identify as the
full subcategory of $\POpd(\SpF)$ spanned by those pinned
$\SpF$-\iopds{}
$(\mathcal{O}, \alpha \colon X \twoheadrightarrow
\mathcal{O}_{\fset{1}}^{\simeq})$ where the map $\alpha$ is an
equivalence, and the full subcategory of \emph{complete} dendroidal
Segal spaces, meaning those
dendroidal Segal spaces whose underlying Segal space is complete in
the sense of Rezk~\cite{RezkCSS}.

We can regard (pinned) \icats{} as (pinned) \iopds{} with only unary
operations, and similarly we can regard Segal spaces over $\simp$ as a
full subcategory of dendroidal Segal spaces (whose values at all
non-unary corollas are $\emptyset$). Most of this section is concerned
with showing that our equivalence and that of \cite{polynomial}
restricts to an equivalence between these full subcategories, thus
producing the following
commutative diagram:
\begin{equation}
  \label{eq:linmnddiag}
  \begin{tikzcd}
    \POpd(\SpF) \arrow{r}{\sim} & \AnMnd &
    \Seg_{\Oop}(\mathcal{S}) \arrow{l}[swap]{\sim} \\
    \PCatI \arrow{r}{\sim} \arrow[hookrightarrow]{u} & \LinMnd \arrow[hookrightarrow]{u} & \Seg_{\Dop}(\mathcal{S})
    \arrow{l}[swap]{\sim} \arrow[hookrightarrow]{u}
  \end{tikzcd}
\end{equation}
Here:
\begin{itemize}
\item $\PCatI$ is the \icat{} of \emph{pinned \icats{}}, meaning 
\icats{} $\mathcal{C}$ equipped with essentially surjective maps of
\igpds{} $X \twoheadrightarrow \mathcal{C}^{\simeq}$.
\item The fully faithful inclusion $\PCatI \hookrightarrow
  \POpd(\SpF)$ is left adjoint to the functor given by
  \[ (\mathcal{O}, p \colon X
  \twoheadrightarrow \mathcal{O}_{\fset{1}}^{\simeq}) \mapsto
  (\mathcal{O}_{\fset{1}}, p).\]
\item $\LinMnd$ is a full subcategory of $\AnMnd$ containing the
  \emph{linear monads} (which preserve all colimits).
\item The fully faithful inclusion $\SegD \hookrightarrow \SegO$ is
  given by left Kan extension along the inclusion $\Dop
  \hookrightarrow \Oop$, and so is left adjoint to the underlying
  Segal space functor.
\end{itemize}
We begin by constructing the left-hand square in
\cref{sec:line-monads-compl}, which amounts to understanding the
restriction of the equivalence $\POpd(\SpF) \simeq \AnMnd$ to the full
subcategory of pinned \iopds{} with only unary operations, which we
identify with pinned \icats{}.  We then construct the right-hand
square in \cref{sec:linmndsegal}; this requires proving that the
equivalence between analytic monads and dendroidal Segal spaces from
\cite{polynomial} restricts to an equivalence between linear monads and
Segal spaces on $\Dop$.  Finally, in
\cref{sec:complete}, we show that the composite equivalence
between pinned \icats{} and Segal spaces restricts to one between
\icats{} and complete Segal spaces, and then extend this to the
operadic setting.

\subsection{Linear monads and $\infty$-categories}
\label{sec:line-monads-compl}

Any \icat{} can be regarded as an \iopd{} with only unary
operations. In this subsection we will first describe this embedding
concretely in terms of $\SpF$-\iopds{}, and then identify the full
subcategory of $\AnMnd$ that corresponds to (pinned) \icats{} when we
restrict the equivalence of \cref{cor:opdmndeq}.

\begin{defn}
  For $\mathcal{C} \in \CatI$, we define $\xF^{\op}_{\mathcal{C}} \to
  \xF^{\op}$ as the cocartesian fibration corresponding to the functor
  \[ j_{*}\mathcal{C} \colon \xF^{\op} \to \CatI \]
  obtained by right Kan extension of $\mathcal{C}$ along the inclusion
  $j \colon \{\fset{1}\} \hookrightarrow \xF^{\op}$. The functor
  $j_{*}\mathcal{C}$ is then given by
  \[ j_{*}\mathcal{C}(\fset{n}) \simeq \Fun(\fset{n}, \mathcal{C})
    \simeq \mathcal{C}^{\times n}.\]
\end{defn}

\begin{lemma}
  The composite $\xF^{\op}_{\mathcal{C}} \to \xF^{\op} \to \SpF$ is a
  $\SpF$-\iopd{}, and $\xF^{\op}_{(\blank)}$ gives a functor $\CatI
  \to \Opd(\SpF)$.
\end{lemma}
\begin{proof}
  We can regard $\xF^{\op}_{(\blank)} \to \SpF$ as the composite
  functor
  \[ \CatI \xto{j_{*}} \Fun(\xF^{\op},\CatI) \simeq
    \Cat_{\infty/\xF^{\op}}^{\txt{coc}} \to
    \Cat_{\infty/\SpF}, \]
  where the last functor is given by composition with the inclusion $i
  \colon \xF^{\op} \to \SpF$ of the inert (\ie{} backwards)
  morphisms. Here $i$ clearly admits cocartesian lifts of inert
  morphisms in $\SpF$, hence so does the composite
  $\xF^{\op}_{\mathcal{C}} \to \SpF$ (given by cocartesian lifts of
  the morphisms in $\xF^{\op}$), and the functor
  $\xF^{\op}_{\mathcal{C}} \to \xF^{\op}_{\mathcal{D}}$ induced by any
  functor $\mathcal{C} \to \mathcal{D}$ in $\CatI$ preserves these, as
  it preserves cocartesian morphisms over $\xF^{\op}$.

  It remains to
  check that $\xF^{\op}_{\mathcal{C}}$ is an \iopd{}. By definition
  we have the Segal condition $(\xF^{\op}_{\mathcal{C}})_{\fset{n}}
  \simeq \prod_{i=1}^{n} \mathcal{C}$. For any objects
  $\fset{n}$, $\fset{m}$ in $\xF$ we have a pullback square
  \[
    \begin{tikzcd}
      \Map_{\xF^{\op}}(\fset{n},\fset{m}) \arrow{r}{\sim} \arrow{d} &
      \prod_{i=1}^{m} \Map_{\xF^{\op}}(\fset{n},\fset{1}) \arrow{d} \\
      \Map_{\SpF}(\fset{n},\fset{m}) \arrow{r}{\sim}  &
      \prod_{i=1}^{m} \Map_{\SpF}(\fset{n},\fset{1}),
    \end{tikzcd}
  \]
  so it suffices to show that for $\phi \colon \fset{n} \to
  \mathcal{C}, \psi \colon \fset{m} \to \mathcal{C}$ we have a
  pullback square
  \[
    \begin{tikzcd}
      \Map_{\xF_{\mathcal{C}}^{\op}}(\phi, \psi) \arrow{r}{\sim} \arrow{d} &
      \prod_{i=1}^{m} \Map_{\xF_{\mathcal{C}}^{\op}}(\phi,\psi(i)) \arrow{d} \\
      \Map_{\xF^{\op}}(\fset{n},\fset{m}) \arrow{r}{\sim} &
      \prod_{i=1}^{m} \Map_{\xF^{\op}}(\fset{n},\fset{1}).
    \end{tikzcd}
  \]
  To see this, observe that for $f \colon \fset{m} \to \fset{n}$, the
  map on fibres over $f$ can be identified with
  \[ \Map_{\Fun(\fset{m}, \mathcal{C})}(\phi\circ f, \psi) \isoto
    \prod_{i=1}^{m} \Map_{\mathcal{C}}(\phi(f(i)), \psi(i)). \qedhere\]
\end{proof}

\begin{lemma}\label{lem:opdpbFop}
  If $\mathcal{O}$ is a $\SpF$-\iopd{}, then there is a natural
  equivalence
  \[ \mathcal{O} \times_{\SpF} \xF^{\op} \isoto \xF^{\op}_{\mathcal{O}_{\fset{1}}}.\]
\end{lemma}
\begin{proof}
  The projection $\mathcal{O} \times_{\SpF} \xF^{\op} \to \xF^{\op}$
  is a cocartesian fibration, corresponding to a functor $\xF^{\op}
  \to \CatI$. Hence the unit of the adjunction $j^{*} \dashv j_{*}$
  corresponds to a morphism $\mathcal{O} \times_{\SpF} \xF^{\op}
  \to \xF^{\op}_{\mathcal{O}_{\fset{1}}}$ of cocartesian fibrations
  over $\xF^{\op}$. To show that this is an equivalence it suffices to
  show that it is an equivalence on the fibre over every $\fset{n} \in
  \xF^{\op}$. To see this we observe that since the functor preserves
  cocartesian morphisms we have a commutative square
  \[
    \begin{tikzcd}
      \mathcal{O}_{\fset{n}} \arrow{r} \arrow{d}{\sim} &
      (\xF^{\op}_{\mathcal{O}_{\fset{1}}})_{\fset{n}} \arrow{d}{\sim} \\
      \prod_{i=1}^{n} \mathcal{O}_{\fset{1}} \arrow{r}{\sim} &
      \prod_{i=1}^{n} (\xF^{\op}_{\mathcal{O}_{\fset{1}}})_{\fset{1}},
    \end{tikzcd}
  \]
  where the vertical maps are equivalences since both sides satisfy
  the Segal condition, and the bottom horizontal map is an equivalence
  since by construction the two fibres over $\fset{1}$ are the same.
\end{proof}

\begin{propn}
  The functor $\xF^{\op}_{(\blank)}$ is fully faithful, and is left
  adjoint to the functor
  \[(\blank)_{\fset{1}} \colon \Opd(\SpF) \to \CatI\] that extracts
  the underlying \icat{} of unary operations in a $\SpF$-\iopd{}.
\end{propn}
\begin{proof}
  Using \cref{lem:opdpbFop}, for $\mathcal{C} \in \CatI,
  \mathcal{O} \in \Opd(\SpF)$ we have natural equivalences
  \[
    \begin{split}
    \Map_{\Opd(\SpF)}(\xF^{\op}_{\mathcal{C}}, \mathcal{O}) & \simeq
    \Map_{\Cat_{\infty/\xF^{\op}}^{\txt{coc}}}(\xF^{\op}_{\mathcal{C}},
    \mathcal{O} \times_{\SpF} \xF^{\op})\\
     &
     \simeq \Map_{\Cat_{\infty/\xF^{\op}}^{\txt{coc}}}(\xF^{\op}_{\mathcal{C}},
     \xF^{\op}_{\mathcal{O}_{\fset{1}}}) \\
     & \simeq \Map_{\Fun(\xF^{\op}, \CatI)}(j_{*}\mathcal{C},
     j_{*}\mathcal{O}_{\fset{1}}) \\
     & \simeq \Map_{\CatI}(j^{*}j_{*}\mathcal{C},
     \mathcal{O}_{\fset{1}}) \\
     & \simeq \Map_{\CatI}(\mathcal{C}, \mathcal{O}_{\fset{1}}).
    \end{split}
  \]
  This shows that $\xF^{\op}_{(\blank)}$ is left adjoint to
  $(\blank)_{\fset{1}}$. Moreover, the unit map
  $\mathcal{C} \to (\xF^{\op}_{\mathcal{C}})_{\fset{1}}$ is an
  equivalence, which implies that $\xF^{\op}_{(\blank)}$ is fully
  faithful.
\end{proof}

\begin{defn}
  We define $\PCatI$ as the pullback $\CatI
  \times_{\mathcal{S}} \Fun([1], \mathcal{S})_{\txt{es}}$ (given by the functors
  $(\blank)^{\simeq}$ and $\ev_{1}$), where $\Fun([1], \mathcal{S})_{\txt{es}}$ is the full subcategory of $\Fun([1], \mathcal{S})$ spanned by  the essentially surjective maps.
\end{defn}

\begin{notation}
  If $\mathcal{C}$ is an \icat{}, we denote the 
  pinned \icat{} \[(\mathcal{C}, \mathcal{C}^{\simeq} \xto{=}
    \mathcal{C}^{\simeq})\]
  by $\mathcal{C}^{\natural}$.
\end{notation}

\begin{observation}\label{obs:mapfromnat}
If $\mathcal{C}$ is an \icat{} and $(\mathcal{D},q \colon X
  \twoheadrightarrow \mathcal{D}^{\simeq})$ is a pinned \icat{},
  we have
  \[ \Map_{\PCatI}(\mathcal{C}^{\natural}, (\mathcal{D},q)) \simeq
    \Map_{\CatI}(\mathcal{C}, \mathcal{D})
    \times_{\Map_{\mathcal{S}}(\mathcal{C}^{\simeq},
      \mathcal{D}^{\simeq})} \Map_{\mathcal{S}}(\mathcal{C}^{\simeq}, X).
  \]  
\end{observation}

\begin{lemma}
  The adjunction $\xF^{\op}_{(\blank)} \dashv (\blank)_{\fset{1}}$
  extends to an adjunction
  \[ \xF^{\op}_{(\blank)} : \PCatI \rightleftarrows \POpd(\SpF) :
    (\blank)_{\fset{1}} \]
  where the left adjoint $\xF^{\op}_{(\blank)}$ is still fully faithful.
\end{lemma}
\begin{proof}
  Immediate from the observation that both adjoints and the (co)unit
  transformations are compatible with the functors $(\blank)^{\simeq}
  \colon \CatI \to \mathcal{S}$ and $(\blank)_{\fset{1}}^{\simeq}
  \colon \Opd(\SpF) \to \mathcal{S}$.
\end{proof}

\begin{propn}
  The \icat{} $\xF^{\op}_{\mathcal{C}}$ has finite products. If
  $\mathcal{D}$ is an \icat{} with finite products, then a
  functor $\xF^{\op}_{\mathcal{C}} \to \mathcal{D}$ preserves products
  \IFF{} it is right Kan extended along the inclusion $\mathcal{C}
  \hookrightarrow \xF^{\op}_{\mathcal{C}}$. Restriction along this
 inclusion therefore gives an equivalence
 \[\Mon_{\xF^{\op}_{\mathcal{C}}}(\mathcal{D}) \isoto
   \Fun(\mathcal{C}, \mathcal{D}).\]
\end{propn}
\begin{proof}
  Since $\xF^{\op}_{\mathcal{C}}$ is a $\SpF$-\iopd{}, we know it has
  finite products by \cref{SpFopdcond}, and by \cref{propn:SpFopdprod}
  a functor $F \colon \xF^{\op}_{\mathcal{C}} \to \mathcal{D}$
  preserves these \IFF{} it is an $\xF^{\op}_{\mathcal{C}}$-monoid,
  \ie{} for every object $X = (x_{1},\ldots,x_{n})$ the
  canonical map
  \begin{equation}
    \label{eq:prodxFopCmap}
  F(X) \to
  \prod_{i=1}^{n}F(x_{i})  
  \end{equation}
  is an equivalence.

  The \icat{}
  \[ \mathcal{C}_{X/} := \mathcal{C}
    \times_{\xF^{\op}_{\mathcal{C}}}
    (\xF^{\op}_{\mathcal{C}})_{X/} \] contains
  the discrete set $S_{X}$ of  cocartesian maps
  $X \to x_{i}$ as a subcategory. We claim the
  inclusion $S_{X} \to
  \mathcal{C}_{X/}$ is coinitial. To see this,
  first observe that the target of the projection 
  \[ \mathcal{C}_{X/} \to \{\fset{1}\}
    \times_{\xF^{\op}} \xF^{\op}_{\fset{n}/}\] is the discrete set of
  morphisms $\fset{1} \to \fset{n}$ in $\xF$, \ie{} the set
  $\fset{n}$. The fibre over an element $i \in \fset{n}$ identifies
  (via the cocartesian morphism $X \to x_{i}$) with
  $\mathcal{C}_{x_{i}/}$, where $x_{i}$ is of course initial. From this the
  criterion of \cite{HTT}*{Theorem 4.1.3.1} immediately implies that
  $S_{X}$ is coinitial.

  It follows that
  pointwise right Kan extensions along the inclusion
  $\mathcal{C} \hookrightarrow \xF^{\op}_{\mathcal{C}}$ exist provided
  the target has finite products, and that a functor
  $F \colon \xF^{\op}_{\mathcal{C}} \to \mathcal{D}$ is a right Kan
  extension of its restriction to $\mathcal{C}$ precisely when the maps
  \cref{eq:prodxFopCmap} are equivalences.
\end{proof}

The composite functor $\PCatI \to \POpd(\SpF) \isoto \AnMnd$ thus
takes the pair $(\mathcal{C}, p \colon X \twoheadrightarrow
\mathcal{C}^{\simeq})$ to the monadic right adjoint
\[ \Mon_{\xF^{\op}_{\mathcal{C}}}(\mathcal{S}) \simeq
  \Fun(\mathcal{C}, \mathcal{S}) \xto{p^{*}} \Fun(X, \mathcal{S}). \]
Our next goal is to identify the image of $\PCatI$ as precisely the
full subcategory of \emph{linear} monads inside $\AnMnd$:
\begin{defn}
  A \emph{linear monad} on $\Fun(X,\mathcal{S})$ is an analytic monad
  such that the underlying endofunctor preserves small colimits.
\end{defn}

\begin{observation}\label{rmk:linearpolydiag}
  Since presheaves on an \icat{} are its free cocompletion, we can
  identify colimit-preserving functors $\Fun(X, \mathcal{S}) \to
  \Fun(Y, \mathcal{S})$ for $X,Y \in \mathcal{S}$ with
  \[ \Fun(X, \Fun(Y, \mathcal{S})) \simeq \Fun(X \times Y,
    \mathcal{S}) \simeq \mathcal{S}_{/X\times Y},\]
  so that a colimit-preserving functor corresponds to a \emph{span}
  \[
    \begin{tikzcd}
      {} & E \arrow{dl}[swap]{f} \arrow{dr}{g} \\
      X & & Y,
    \end{tikzcd}
  \]
  with this span corresponding to the functor $g_{!}f^{*}$. In terms
  of polynomial diagrams, this means a colimit-preserving functor
  corresponds to a diagram
  \[ X \xfrom{s} E \xto{p} B \xto{t} Y,\]
  where $p$ is an \emph{equivalence}.
\end{observation}

As a special case of \cref{propn:monadcolim}, we have:
\begin{cor}\label{linmndcocts}
  A monad $T$ on $\Fun(X, \mathcal{S})$ is linear \IFF{} for the
  corresponding monadic adjunction $F_{T} \dashv U_{T}$, the \icat{}
  $\Alg(T)$ has small colimits and the
  right adjoint
  $U_{T} \colon \Alg(T) \to \Fun(X, \mathcal{S})$ preserves these. \qed
\end{cor}

\begin{propn}
  Suppose $T$ is a linear monad on $\Fun(X, \mathcal{S})$, and let
  $\mathcal{C}^{\op}$ be the full subcategory of $\Alg(T)$ spanned by
  the objects $F_{T}\Yo_{X}(x)$ for $x \in X$. Then the restricted
  Yoneda embedding $\Alg(T) \to \Fun(\mathcal{C}, \mathcal{S})$ is an
  equivalence.
\end{propn}
\begin{proof}
  We apply \cite{polynomial}*{Lemma 3.3.11}, for which it suffices
  to check that the objects of $\mathcal{C}^{\op}$ are completely
  compact and jointly conservative. We have a natural equivalence
  \[ \Map_{\Alg(T)}(F_{T}\Yo_{X}(x), A) \simeq \Map_{\Fun(X,
      \mathcal{S})}(\Yo_{X}(x), U_{T}A) \simeq (U_{T}A)(x),\]
  from which we see that $F_{T}\Yo_{X}(x)$ is completely compact since
  $\Yo_{X}(x)$ is completely compact in $\Fun(X, \mathcal{S})$ and
  $U_{T}$ preserves small colimits by  \cref{linmndcocts}. Moreover,
  the objects of $\mathcal{C}^{\op}$ are jointly conservative since
  $U_{T}$ is conservative and equivalences in $\Fun(X, \mathcal{S})$
  are detected by evaluation at the points of $X$.
\end{proof}

\begin{cor}
  An analytic monad $T$ is linear \IFF{} it is in the image of
  $\PCatI$. \qed
\end{cor}

\begin{cor}
  We have a commutative square
  \[
    \begin{tikzcd}
      \POpd(\SpF) \arrow{r}{\sim} & \AnMnd \\
      \PCatI \arrow[hookrightarrow]{u}{\xF^{\op}_{(\blank)}}
      \arrow{r}{\sim} & \LinMnd, \arrow[hookrightarrow]{u}
    \end{tikzcd}
  \]
  where the bottom horizontal functor takes a pinned \icat{} $(\mathcal{C}, p \colon X
  \twoheadrightarrow \mathcal{C}^{\simeq})$ to the monadic right
  adjoint
  $\Fun(\mathcal{C}, \mathcal{S}) \xto{p^{*}} \Fun(X,
    \mathcal{S})$. \qed
\end{cor}

\subsection{Linear monads and Segal spaces}\label{sec:linmndsegal}
Our goal in this section is to show that the equivalence between
analytic monads and dendroidal Segal spaces from \cite{polynomial}
restricts to an equivalence between linear monads and Segal spaces.

In order to do this, we will briefly revisit each step in the
comparison and discuss how it restricts to the
linear case; we refer the reader to \cite{polynomial} for full details.

Firstly, by giving an alternative description of the \icat{} $\AnMnd$
we constructed (in \cite{polynomial}*{\S 5.1}) a forgetful functor
$\Uan \colon \AnMnd \to \AnEnd$ where $\AnEnd$ is an \icat{}
of \emph{analytic endofunctors} on \icats{} of the form
$\mathcal{S}_{/X}$; an object of $\AnEnd$ is a space $X$ together with
an analytic endofunctor
$T \colon \mathcal{S}_{/X} \to \mathcal{S}_{/X}$, and a morphism
$(X,T) \to (Y,S)$ is a morphism of spaces $f \colon X \to Y$ together
with a cartesian transformation $T f^{*} \to f^{*} S$.

\begin{defn}
  We define $\LinEnd$ as the full subcategory of
$\AnEnd$ containing the linear (\ie{} colimit-preserving)
endofunctors. By construction we then have a pullback square
\[
  \begin{tikzcd}
    \LinMnd \arrow[hookrightarrow]{r} \arrow{d}{\Ulin} &
    \AnMnd \arrow{d}{\Uan} \\
    \LinEnd \arrow[hookrightarrow]{r} & \AnEnd.
  \end{tikzcd}
\]
\end{defn}

Next, we defined the category $\bbOint$ of trees and inert morphisms
(or subtree inclusions) as a full subcategory of $\AnEnd$. Here we
regard a tree $T$ as the endofunctor corresponding to the polynomial 
diagram
\[ E \xfrom{s} V_{*} \xto{p} V \xto{t} E,\]
where $E$ is the set of edges of $T$, $V$ is the set of vertices, and
$V_{*}$ is the set of vertices with a marked incoming edge. The map
$s$ is the projection $(v,e) \mapsto e$ and $p$ is the projection
$(v,e) \mapsto v$, while $t$ takes a vertex $v$ to its unique outgoing
edge. Here the fibre of $p$ at $v$ is the set of incoming edges of
$v$, so the tree $T$ is linear in the sense that it has only unary
vertices \IFF{} $p$ is an equivalence, which we saw in 
\cref{rmk:linearpolydiag} is equivalent to the corresponding functor
being linear. It is easy to see that the full subcategory of $\bbO^{\txt{int}}$ on
the linear trees is $\simp^{\txt{int}}$, so we get a pullback square
\[
  \begin{tikzcd}
    \simp^{\txt{int}} \arrow[hookrightarrow]{d}
    \arrow[hookrightarrow]{r} & \LinEnd \arrow[hookrightarrow]{d} \\
    \bbOint \arrow[hookrightarrow]{r} & \AnEnd.
  \end{tikzcd}
\]
Let us write $\bbOel$ for the full subcategory of $\bbOint$ on the
edge $\eta$ and the corollas $C_{n}$ (with $C_{n}$ having $n$ leaves
for $n = 0,1,\ldots$). Then we define $\Del$ as the intersection of
$\bbOel$ with $\Dint$, so that $\Del$ has objects $\eta = [0]$ and
$C_{1} = [1]$.

In \cite{polynomial}*{\S 3.3} we showed that the restricted Yoneda
embedding $\AnEnd \to \mathcal{P}(\bbOel)$ is an equivalence, while
$\AnEnd \to \mathcal{P}(\bbOint)$ is fully faithful with image the
functors that satisfy the dendroidal Segal condition, or equivalently
are right Kan extended from $\bbOel$. To deduce an equivalent
statement for linear endofunctors we first prove the following
characterization:

\begin{propn}\label{propn:anendlinearchar}
  The following are equivalent for an object $F \in \AnEnd$.
  \begin{enumerate}[(i)]
  \item $F$ is linear.
  \item $\Map_{\AnEnd}(C_{n},F) \simeq \emptyset$ for $n \neq 1$.    
  \item $\Map_{\AnEnd}(T, F) \simeq \emptyset$ for every tree $T \in
    \bbOint$ which is not linear.
  \item $\Map_{\AnEnd}(\blank,F)|_{\bbOint} \in \mathcal{P}(\bbOint)$
    is left Kan extended from $\Dintop$.
  \item $\Map_{\AnEnd}(\blank,F)|_{\bbOel} \in \mathcal{P}(\bbOel)$
    is left Kan extended from $\Delop$.
  \end{enumerate}
\end{propn}
\begin{proof}
  Suppose $F$ corresponds to the polynomial diagram
  \[ X \xfrom{s} E \xto{p} B \xto{t} X,\]
  where $p$ has finite discrete fibres. The map $p$ is then classified
  by a map $\pi \colon B \to \xF^{\simeq}$, and by \cite{polynomial}*{Lemma
    3.3.6} the space $\Map_{\AnEnd}(C_{n}, F)$ is equivalent to
  the fibre of $\pi$ at an $n$-element set. This fibre is empty \IFF{}
  no fibres of $p$ are $n$-element sets, so that condition (ii) is
  equivalent to all fibres of $p$ having size $1$, \ie{} to $p$ being
  an equivalence, which by \cref{rmk:linearpolydiag} corresponds to
  $F$ being linear.

  Condition (ii) is a special case of (iii), but it also implies (iii)
  since every tree is a colimit in $\AnEnd$ of its edges and corollas
  by \cite{polynomial}*{Lemma 3.3.5}, and so $\Map_{\AnEnd}(T, F)$
  must be empty whenever $T$ contains a non-unary vertex.

  To see that (ii) and (iii) are equivalent to (iv) and (v),
  respectively, it suffices to note that there are no maps in
  $\bbOint$ from a non-linear tree to a linear one.
\end{proof}

\begin{cor}\label{cor:LinEndsquares}
  We have commutative squares
  \[
    \begin{tikzcd}
     \LinEnd \arrow{r}{\sim} \arrow[hookrightarrow]{d} &
     \PSh(\Del) \arrow[hookrightarrow]{d} \\
     \AnEnd \arrow{r}{\sim} & \PSh(\bbOel),
   \end{tikzcd}
   \qquad
    \begin{tikzcd}
     \LinEnd \arrow{r}{\sim} \arrow[hookrightarrow]{d} &
     \PSh_{\Seg}(\Dint) \arrow[hookrightarrow]{d} \\
     \AnEnd \arrow{r}{\sim} & \PSh_{\Seg}(\bbOint),
   \end{tikzcd}
 \]
 where in both the right-hand vertical arrow is a left Kan extension.
\end{cor}

\begin{proof}
  Let $\lambda,
  \delta^{\txt{el}}, \omega^{\txt{el}}, L^{\txt{el}}$ denote the
  inclusions $\LinEnd \hookrightarrow \AnEnd$, $\Del \hookrightarrow
  \LinEnd$, $\bbOel \hookrightarrow \AnEnd$, $\Del \hookrightarrow
  \bbOel$, respectively. From the commutative square
  \[
    \begin{tikzcd}
      \Del \arrow[hookrightarrow]{r}{\delta^{\txt{el}}}
      \arrow[hookrightarrow]{d}{L^{\txt{el}}} & \LinEnd
      \arrow[hookrightarrow]{d}{\lambda} \\
      \bbOel \arrow[hookrightarrow]{r}{\omega^{\txt{el}}} & \AnEnd
    \end{tikzcd}
  \]
  we obtain a commutative diagram
  \[
    \begin{tikzcd}
      \LinEnd \arrow[hookrightarrow]{r}{\Yo_{\LinEnd}} \arrow[hookrightarrow]{d}{\lambda} & \PSh(\LinEnd)
      \arrow{r}{\delta^{\txt{el},*}} \arrow{d}{\lambda_{!}} &
      \PSh(\Del) \arrow[hookrightarrow]{d}{L^{\txt{el}}_{!}} \arrow[Rightarrow]{dl} \\
      \AnEnd \arrow[hookrightarrow]{r}{\Yo_{\AnEnd}} & \PSh(\AnEnd)
      \arrow{r}{\omega^{\txt{el},*}} & \PSh(\bbOel),
    \end{tikzcd}
  \]
  where the left square comes from the Yoneda embedding and the right
  square is the mate of the square of restriction functors from square
  above. To obtain the first commutative square it then suffices to
  show that the composite lax square actually commutes, \ie{} that for
  a linear endofunctor $F$ the map
  $L_{!}^{\txt{el}}\delta^{\txt{el},*}\Yo_{\LinEnd}F \to
  \omega^{\txt{el},*}\Yo_{\AnEnd}\lambda F$ is an equivalence. We can
  check this by evaluating at every object $T \in \bbOel$, where we
  get the canonical map
  \[ \colim_{X \in (\Del)^{\op}_{/T}} \Map(\delta^{\txt{el}}(X), F)
    \to \Map(\omega^{\txt{el}}(T), \lambda F);\]
  if $T$ lies in $\Del$ then this is an equivalence since
  $(\Del)^{\op}_{/T}$ has a terminal object, and if $T$ does not lie
  in $\Del$ it is an equivalence because $(\Del)^{\op}_{/T}$ is empty
  and the right-hand side is $\emptyset$ because $F$ is linear. We
  thus have a commutative square
  \[
    \begin{tikzcd}
     \LinEnd \arrow{r} \arrow[hookrightarrow]{d} &
     \PSh(\Del) \arrow[hookrightarrow]{d}{L^{\txt{el}}_{!}} \\
     \AnEnd \arrow{r}{\sim} & \PSh(\bbOel),
   \end{tikzcd}
 \]
 and to obtain the first square it remains only to show that the top
 horizontal morphism is an equivalence. To see this, we first observe
 that this functor is automatically fully faithful (since essentially surjective
 and fully faithful functors are a factorization system on
 $\CatI$). It then suffices to observe that the image of $\LinEnd$ in
 $\PSh(\bbOel)$ consists precisely of the functors left Kan extended from
 $\Del$ by \cref{propn:anendlinearchar}. We omit the argument for the
 second square, since it is essentially the same.
\end{proof}

In \cite{polynomial}*{\S 5.2}, we proved that
$\Uan$  has a left adjoint $\Fan \colon \AnEnd
\to \AnMnd$, taking an analytic endofunctor to the free monad on
it. We can define $\bbO$ as the full subcategory of $\AnMnd$
containing the free monads on the objects of $\bbO^{\xint}$; $\simp
\subseteq \bbO$
is then the full subcategory of $\AnMnd$ containing the free monads on
linear trees.

\begin{lemma}
  $\Fan$ restricts to a functor $\Flin \colon \LinEnd \to \LinMnd$,
  left adjoint to $\Ulin$.
\end{lemma}
\begin{proof}
  If $P$ is a linear endofunctor, then we must show that the free analytic monad
  $F_{\txt{an}}P$ is a linear monad. This amounts to checking that the
  underlying endofunctor of $F_{\txt{an}}P$ is
  linear. This follows from the explicit description of this
  endofunctor in \cite{polynomial}*{Theorem 5.2.4}, since the space of
  maps from any non-linear tree to $P$ is empty if $P$ is a linear endofunctor.
\end{proof}

\begin{propn}\label{propn:linearmndcond}
  The following are equivalent for an analytic monad $M \in \AnMnd$.
  \begin{enumerate}[(i)]
  \item $M$ is linear.
  \item $\Map_{\AnMnd}(C_{n}, M) \simeq \emptyset$ for $n \neq 1$.
  \item $\Map_{\AnMnd}(T, M) \simeq \emptyset$ for every $T \in \bbO$
    which is not linear.
  \item $\Map_{\AnMnd}(\blank, M)|_{\bbO^{\op}} \in \PSh(\bbO)$
    is left Kan extended from $\Dop$.
  \end{enumerate}
\end{propn}
\begin{proof}
  By definition, $M$ is linear precisely if $\Uan M$ is linear;
  since the elements of $\bbO$ are free, the equivalence of the first
  three conditions follows immediately from
  \cref{propn:anendlinearchar}. The last two
  conditions are then equivalent because a presheaf on $\bbO$ is left
  Kan extended from its restriction to $\simp^{\op}$ precisely when
  its value at all non-linear trees is $\emptyset$, because there are
  no maps from a non-linear tree to a linear one in $\bbO$.
\end{proof}

In \cite{polynomial}*{\S 5.3} we proved as our main theorem that the
restricted Yoneda embedding $\AnMnd \to \PSh(\bbO)$ fits in a pullback
square
\[
  \begin{tikzcd}
    \AnMnd \arrow[hookrightarrow]{r} \arrow{d}{U_{\txt{an}}}  &
    \PSh(\bbO) \arrow{d} \\
    \AnEnd \arrow[hookrightarrow]{r} & \PSh(\bbO^{\xint}),
  \end{tikzcd}
\]
and so identifies $\AnMnd$ with the full subcategory
$\SegO \subseteq \PSh(\bbO)$. Together with \cref{propn:linearmndcond}
this gives the following analogue of \cref{cor:LinEndsquares} (with
essentially the same proof):
\begin{cor}\label{cor:LinMndsquare}
  There is a commutative square
    \[
    \begin{tikzcd}
     \LinMnd \arrow{r}{\sim} \arrow[hookrightarrow]{d} &
     \PSh_{\Seg}(\simp) \arrow[hookrightarrow]{d} \\
     \AnMnd \arrow{r}{\sim} & \PSh_{\Seg}(\bbO),
   \end{tikzcd}
 \]
 where the right-hand vertical arrow is a left Kan extension. \qed
\end{cor} 

Our final goal in this section is to combine this square and that from \cref{cor:LinEndsquares} into a commutative cube:
\begin{propn}
  There is a commutative diagram
  \[
    \begin{tikzcd}[row sep=small,column sep=small]
      \LinMnd \arrow{rr}{\sim} \arrow[hookrightarrow]{dr}
      \arrow{dd} & &
      \PSh_{\Seg}(\simp) \arrow[hookrightarrow]{dr}
      \arrow{dd} \\
      & \AnMnd \arrow[crossing over]{rr}[near start]{\sim} &  &
      \PSh_{\Seg}(\bbO) \arrow{dd} \\
      \LinEnd \arrow{rr}[near start]{\sim} \arrow[hookrightarrow]{dr} & &
      \PSh_{\Seg}(\Dint) \arrow[hookrightarrow]{dr} \\
      & \AnEnd \arrow[leftarrow,crossing over]{uu} \arrow{rr}{\sim} & &  \PSh_{\Seg}(\bbOint) \\
    \end{tikzcd}
  \]
  where all the squares are cartesian.
\end{propn}
\begin{proof}
  We have a commutative cube
  \[
    \begin{tikzcd}[row sep=small,column sep=small]
      \Dint \arrow[hookrightarrow]{rr} \arrow[hookrightarrow]{dr}
      \arrow{dd} & & \LinEnd
       \arrow[hookrightarrow]{dr}
      \arrow{dd}[near start]{\Flin} \\
      & \bbOint \arrow[crossing over,hookrightarrow]{rr} &  &
      \AnEnd \arrow{dd}{\Fan} \\
      \simp \arrow[hookrightarrow]{rr} \arrow[hookrightarrow]{dr} & &
      \LinMnd \arrow[hookrightarrow]{dr} \\
      & \bbO \arrow[leftarrow,crossing over]{uu} \arrow[hookrightarrow]{rr} & & \AnMnd. \\
    \end{tikzcd}
  \]
  Taking presheaves and mates we get (since mates in $\CatI$ are
  functorial, as was proved in \cite{HHLN1}) the following diagram, where all except the
  front and back squares are a priori lax:
  \[
    \begin{tikzcd}[row sep=small,column sep=small]
      \PSh(\LinMnd) \arrow{rr} \arrow[hookrightarrow]{dr}
      \arrow{dd}{\Flin^{*}} & & \PSh(\simp)
       \arrow[hookrightarrow]{dr}
      \arrow{dd} \\
      & \PSh(\AnMnd) \arrow[crossing over]{rr} &  &
      \PSh(\bbO) \arrow{dd} \\
      \PSh(\LinEnd) \arrow{rr} \arrow[hookrightarrow]{dr} & &
      \PSh(\Dint) \arrow[hookrightarrow]{dr} \\
      & \PSh(\AnEnd) \arrow[leftarrow,crossing over]{uu}[near start]{\Fan^{*}} \arrow{rr} & & \PSh(\bbOint). \\
    \end{tikzcd}
  \]
  The functor $\PSh^{*}$ preserves adjunctions, so we can identify
  $\Fan^{*}$ as the left adjoint of $\Uan^{*}$. Moreover, $\PSh^{*}$
  preserves mate squares, so that the mate in the vertical direction
  of the square
  \[
    \begin{tikzcd}
      \PSh(\AnMnd) \arrow{r} \arrow{d}{\Fan^{*}} & \PSh(\LinMnd) \arrow{d}{\Flin^{*}} \\
      \PSh(\AnEnd) \arrow{r} & \PSh(\LinEnd) 
    \end{tikzcd}
  \]
  is the commutative square
  \[
    \begin{tikzcd}
      \PSh(\AnEnd) \arrow{r} \arrow{d}{\Uan^{*}} & \PSh(\LinEnd)
      \arrow{d}{\Ulin^{*}} \\      
      \PSh(\AnMnd) \arrow{r}  & \PSh(\LinMnd)
    \end{tikzcd}
  \]
  obtained by applying $\PSh^{*}$ to the square
  \[
    \begin{tikzcd}
      \LinMnd \arrow[hookrightarrow]{r} \arrow{d}{\Ulin} & \AnMnd \arrow{d}{\Uan} \\
      \LinEnd \arrow[hookrightarrow]{r} & \AnEnd.
    \end{tikzcd}
  \]
  The left-hand square in our cube above is therefore obtained by
  taking mates in both vertical and horizontal directions for this
  last square, and so it is precisely the corresponding commutative
  square of left adjoints (and so is in particular not actually lax). We can therefore compose our cube with the
  following commutative cube arising from the naturality of the Yoneda
  embedding:
  \[
    \begin{tikzcd}[row sep=small,column sep=small]
      \LinMnd  \arrow[hookrightarrow]{rr} \arrow[hookrightarrow]{dr}
      \arrow{dd}{\Ulin} & & \PSh(\LinMnd)
       \arrow[hookrightarrow]{dr}
      \arrow{dd}[near start]{(\Ulin)_{!}} \\
      & \AnMnd \arrow[crossing over,hookrightarrow]{rr} &  &
      \PSh(\AnMnd) \arrow{dd}{(\Uan)_{!}} \\
      \LinEnd \arrow[hookrightarrow]{rr} \arrow[hookrightarrow]{dr} & &
      \PSh(\LinEnd) \arrow[hookrightarrow]{dr} \\
      & \AnEnd \arrow[leftarrow,crossing over]{uu}[near start]{\Uan} \arrow[hookrightarrow]{rr} & & \PSh(\AnEnd).
    \end{tikzcd}
  \]
  This produces a diagram of the form
  \[
    \begin{tikzcd}[row sep=small,column sep=small]
      \LinMnd \arrow{rr} \arrow[hookrightarrow]{dr}
      \arrow{dd} & &
      \PSh(\simp) \arrow[hookrightarrow]{dr}
      \arrow{dd} \\
      & \AnMnd \arrow[crossing over]{rr} &  &
      \PSh(\bbO) \arrow{dd} \\
      \LinEnd \arrow{rr} \arrow[hookrightarrow]{dr} & &
      \PSh(\Dint) \arrow[hookrightarrow]{dr} \\
      & \AnEnd \arrow[leftarrow,crossing over]{uu} \arrow{rr} & &  \PSh(\bbOint) \\
    \end{tikzcd}
  \]
  where the top, bottom and right-hand squares are potentially lax. To
  complete the proof it remains to show that these squares actually
  commute. For the top and bottom this follows from
  \cref{cor:LinMndsquare} and 
  \cref{cor:LinEndsquares}, respectively, while for the right-hand
  square it is immediate (since both left Kan extensions amount to
  extending the functors by $\emptyset$ on non-linear trees). 
\end{proof}

This completes the construction of the diagram \cref{eq:linmnddiag}.

\begin{remark}
  Combining our work so far in this section, we have in particular 
  an equivalence $\PCatI \simeq \Seg_{\Dop}(\mathcal{S})$ between
  pinned \icats{} and (not necessarily complete) Segal spaces in the
  sense of Rezk~\cite{RezkCSS}. Such an equivalence has previously
  been proved by Ayala and Francis \cite{AyalaFrancisFlagged} (as the
  case $n=1$ of a more general equivalence between \emph{flagged
    $(\infty,n)$-categories} and Rezk's Segal $\Theta_{n}$-spaces
  \cite{RezkThetaN}).
\end{remark}

\subsection{Completeness}\label{sec:complete}

Our goal in this final section is to understand the composite of the equivalences
\[ \PCatI \isoto \LinMnd \isofrom \Seg_{\Dop}(\mathcal{S}),\]
and to show that it restricts to the standard equivalence between
\icats{} and complete Segal spaces. Using this, we can then complete
the proof that $\SpF$-\iopds{} are equivalent to complete dendroidal
Segal spaces.

\begin{lemma}
  The composite functor $\PCatI \isoto \LinMnd \xto{U_{\txt{lin}}}
  \LinEnd$ takes a pinned \icat{} $(\mathcal{C}, p \colon X \twoheadrightarrow
  \mathcal{C}^{\simeq})$, to the span
  \[ S(\mathcal{C},p) \quad :=\quad X \from \Map([1],
    \mathcal{C})\times_{\mathcal{C}^{\simeq} \times
      \mathcal{C}^{\simeq}} X \times X \to X,\]
  where the pullback is along $(\ev_{0},\ev_{1}) \colon \Map([1],
  \mathcal{C}) \to \mathcal{C}^{\simeq} \times \mathcal{C}^{\simeq}$
  and two copies of $p$.
\end{lemma}
\begin{proof}
  We first identify the functor $X \times X \to \mathcal{S}$
  corresponding to the endofunctor
  \[ \Fun(X, \mathcal{S}) \xto{p_{!}} \Fun(\mathcal{C}, \mathcal{S})
    \xto{p^{*}} \Fun(X,\mathcal{S}).\]
  This is the functor
  \[ 
    \begin{split}
     (x,x') \mapsto
    \Map_{\Fun(X,\mathcal{S})}(\Yo_{X}(x), p^{*}p_{!}\Yo_{X}(x')) & \simeq
    \Map_{\Fun(\mathcal{C},\mathcal{S})}(p_{!}\Yo_{X}(x), 
    p_{!}\Yo_{X}(x')) \\ & \simeq \Map_{\mathcal{C}}(p(x), p(x')). 
    \end{split}
  \]
  To find the corresponding span we now observe that $\Map([1],
  \mathcal{C}) \to \mathcal{C}^{\simeq} \times \mathcal{C}^{\simeq}$
  is the fibration for the mapping space functor restricted to the
  underlying \igpd{} of $\mathcal{C}$ (since $\Map([1], \mathcal{C})$ is the
  underlying \igpd{} of the twisted arrow \icat{}), and composition of
  functors corresponds to pulling back fibrations.
\end{proof}

\begin{propn}\label{propn:simpemb}
  Under the composite equivalence $\epsilon \colon \SegD \isoto \LinMnd \isoto
  \PCatI$, the full subcategory $\simp$ in $\SegD$ corresponds to the
  standard embedding of $\simp$ in $\CatI \subseteq \PCatI$.
\end{propn}
\begin{proof}
  Since the resulting functor $\simp \to \PCatI$ is fully faithful, it almost
  suffices to check that it does the right thing on objects. Moreover,
  since we have the colimit decomposition $[n] \simeq [1] \amalg_{[0]}
  \cdots \amalg_{[0]} [1]$ in $\SegD$, and hence for the image in $\PCatI$,
  it is enough to look at the objects $[0]$ and $[1]$. 

  By construction, $[0]$ and $[1]$ correspond to the free linear
  monads on the spans
  \[ \eta \quad=\quad \fset{1} \from \emptyset \to \emptyset \to \fset{1},\]
  \[ C_{1} \quad=\quad \fset{2} \from \{1\} \to \{2\} \to \fset{2}. \]
  By \cite{polynomial}*{Lemma
    3.3.6} we then have for any span $T = X \from Y \to X$ that
  \[ \Map_{\LinEnd}(\eta, T) \simeq \Map_{\mathcal{S}}(\fset{1}, X)
    \simeq X,\]
  \[ \Map_{\LinEnd}(C_{1}, T) \simeq Y.\]
  Thus using \cref{obs:mapfromnat} we have
  \[
    \begin{split}
     \Map([0], \epsilon^{-1}(\mathcal{C},p)) & \simeq \Map_{\LinEnd}(\eta,
    S(\mathcal{C}, p)) \simeq X \\ &  \simeq 
    \Map_{\PCatI}([0]^{\natural}, (\mathcal{C},p)), \\
    \Map([1], \epsilon^{-1}(\mathcal{C},p)) & \simeq \Map_{\LinEnd}(C_{1},
    S(\mathcal{C}, p)) \simeq \Map([1],
    \mathcal{C})\times_{\mathcal{C}^{\simeq} \times
      \mathcal{C}^{\simeq}} X \times X  \\ &
    \simeq \Map_{\PCatI}([1]^{\natural}, (\mathcal{C},p)),
    \end{split}
  \]
  as required. This pins down the functor $\simp \to \PCatI$ up to
  automorphisms of $\simp$, of which the only non-trivial one is the
  order-reversing functor $\op \colon \simp \isoto \simp$. To see that
  our functor is not reversing the order we need only observe that the
  two inclusions of $[0]$ in $[1]$ indeed occur in the expected order.
\end{proof}

\begin{observation}\label{rmk:SegspfromPCat}
  It follows from \cref{propn:simpemb}  that the Segal space corresponding to a pinned \icat{}
  $(\mathcal{C}, p \colon X \twoheadrightarrow \mathcal{C}^{\simeq})$
  under the equivalence $\epsilon$ is given by
  \[ [n] \mapsto \Map_{\PCatI}([n]^{\natural}, (\mathcal{C}, p))
    \simeq \Map_{\CatI}([n], \mathcal{C})
    \times_{(\mathcal{C}^{\simeq})^{\times n}} X^{\times n}.\]  
\end{observation}

\begin{propn}
  An object of $\PCatI$ lies in the full subcategory $\CatI$ \IFF{} it
  is local with respect to the map
  \[ E:= ([0], \fset{2} \twoheadrightarrow *) \to [0]^{\natural}.\]
\end{propn}
\begin{proof}
  For $(\mathcal{C}, p \colon X \twoheadrightarrow
  \mathcal{C}^{\simeq})$, we have a natural equivalence
  \[ \Map_{\PCatI}(E, (\mathcal{C}, p)) \simeq X
    \times_{\mathcal{C}^{\simeq}} X, \]
  under which the map $X \simeq \Map([0]^{\natural}, (\mathcal{C},p))
  \to \Map(E, (\mathcal{C},p))$ corresponds to the diagonal. But the
  diagonal $X \to X \times_{\mathcal{C}^{\simeq}} X$ is an
  equivalence \IFF{} the map $p$ is a monomorphism by
  \cite{HTT}*{Lemma 5.5.6.15}. Since $p$ is surjective by assumption,
  if it is a monomorphism then it is necessarily an equivalence, which
  completes the proof.
\end{proof}

\begin{cor}\label{cor:lincompeq}
  Under the equivalence $\SegD \isoto \LinMnd \isoto
  \PCatI$, the
  full subcategory $\CatI \subseteq \PCatI$ is identified with the full
  subcategory $\CSegD$ of \emph{complete} Segal spaces.
\end{cor}
\begin{proof}
  From the formula of \cref{rmk:SegspfromPCat}, the object $E := ([0],
  \fset{2} \twoheadrightarrow *)$ corresponds to the simplicial set
  \[ \Map([n], [0]) \times_{*^{\times n}} \fset{2}^{\times n} \simeq
    \fset{2}^{\times n}.\]
  This is precisely the nerve of the contractible groupoid with two
  objects, and by definition $\CSegD$ is the full subcategory of
  $\SegD$ of objects that are local with respect to the map from this
  to the terminal object. 
\end{proof}

\begin{cor}
  Under the equivalence $\SegO \simeq \AnMnd \simeq \POpd(\SpF)$, the
  full subcategory $\Opd(\SpF) \subseteq \POpd(\SpF)$ is identified
  with the full subcategory $\CSegO$ of \emph{complete} dendroidal
  Segal spaces.
\end{cor}
\begin{proof}
  We have constructed a commutative diagram
  \[
    \begin{tikzcd}
      \SegD \arrow{r}{\sim} \arrow[hookrightarrow]{d} & \LinMnd
      \arrow[hookrightarrow]{d} & \PCatI \arrow[hookrightarrow]{d}
      \arrow{l}[swap]{\sim} \\
      \SegO \arrow{r}{\sim}  & \AnMnd
       & \POpd(\SpF)\arrow{l}[swap]{\sim}.
    \end{tikzcd}
  \]
  Here the left and right vertical arrows have left adjoints, given
  respectively by restriction along the inclusion $\simp
  \hookrightarrow \bbO$ and by extracting the fibre over
  $\fset{1}$. We therefore also have a commutative square
  \[
    \begin{tikzcd}
      \SegO \arrow{r}{\sim} \arrow{d} & \POpd(\SpF)  \arrow{d} \\
      \SegD \arrow{r}{\sim} & \PCatI
    \end{tikzcd}
  \]
  with these left adjoints. Here $\CSegO$ can be defined as the full subcategory of
  $\SegO$ comprising those objects whose image in $\SegD$ lies in
  $\CSegD$. From \cref{cor:lincompeq} it follows that under our
  equivalence this full subcategory is identified with the full
  subcategory of $\POpd(\SpF)$ that contains the objects whose image
  in $\PCatI$ lies in $\CatI$, which is precisely $\Opd(\SpF)$.
\end{proof}

\begin{remark}
  We can also explicitly identify the full subcategory of $\AnMnd$
  that corresponds to $\Opd(\SpF)$ and $\CSegO$ under our
  equivalences: this contains the analytic monads $T$ on $\Fun(X,
  \mathcal{S})$ that are \emph{complete} in the sense that 
  the induced essentially surjective map $X \to
  \mathcal{L}(T)_{\fset{1}}^{\simeq}$ is an equivalence.
  (Here $\mathcal{L}(T)_{\fset{1}}^{\simeq}$ is the same object as
  $\mathcal{U}(T)$ in the notation of \cite{patterns1}*{\S 15}, so this
  definition of complete analytic monads agrees with the notion of
  completeness defined there.)
\end{remark}

\end{document}